\documentclass[12pt]{article}
\usepackage{amssymb,amsmath,amsthm,secdot}
\usepackage{bbm}
\usepackage[shortlabels]{enumitem}

\sectiondot{subsection}
\usepackage{xcolor}
\usepackage{authblk}
\DeclareMathOperator{\I}{\mathbbm{1}}%

\DeclareMathOperator{\Law}{Law}%

\DeclareMathOperator{\curl}{curl}
\DeclareMathOperator{\Span}{Span}
\DeclareMathOperator{\Range}{Range}

\newcommand{\norm}[1] {|\!|\!|#1|\!|\!|}
\usepackage{color}

\newcommand{\blue}{}

\def\E{\hskip.15ex\mathsf{E}\hskip.10ex}
\def\P{\mathsf{P}}
\def\Q{\mathsf{Q}}

\def\eps{\varepsilon}
\def\phi{\varphi}

\newtheoremstyle{Assump}%
  {3pt}
  {3pt}
  {\itshape}
  {}
  {\bfseries}
  {.}
  {.5em}
  {\thmname{#1} \thmnumber{#2} \thmnote{\normalfont#3}}

\newtheorem{Theorem} {Theorem}[section]
\newtheorem{Lemma}[Theorem]{Lemma}
\newtheorem{Proposition}[Theorem]{Proposition}

\theoremstyle{definition}
\theoremstyle{definition}\newtheorem{Remark}[Theorem]{Remark}
\theoremstyle{definition}\newtheorem{Definition}{Definition}[section]

\newtheorem{innercustomthm}{Assumption}
\newenvironment{Assumption}[1]
  {\renewcommand\theinnercustomthm{#1}\innercustomthm}
  {\endinnercustomthm}

\renewcommand{\theequation}{\thesection.\arabic{equation}}
\numberwithin{equation}{section}

\renewcommand{\ge}{\geqslant}
\renewcommand{\le}{\leqslant}
\renewcommand{\leq}{\le}
\renewcommand{\geq}{\ge}

\newcommand{\nn}{\nonumber}
\newcommand{\wt}{\widetilde}
\newcommand{\wh}{\widehat}
\newcommand{\dd}{d}

\renewcommand{\d}{\partial}
\newcommand{\B}{\mathcal{B}}
\newcommand{\C}{\mathcal{C}}
\newcommand{\D}{D}

\newcommand{\N}{\mathbb{N}}
\newcommand{\R}{\mathbb{R}}

\newcommand{\Z}{\mathbb{Z}}

\newcommand{\grad}{\nabla}

\newcommand{\evalat}[1]{_{\mbox{${\mkern1.5mu}\vert_{#1}$}}\!}

\newcommand{\be}{\begin{equation}}
\newcommand{\ee}{\end{equation}}
\newcommand{\ba}{\begin{aligned}}
\newcommand{\ea}{\end{aligned}}

\renewcommand{\u}{\mathbf{u}}
\newcommand{\bx}{\mathbf{x}}
\newcommand{\bX}{\mathbf{X}}
\newcommand{\by}{\mathbf{y}}
\newcommand{\bY}{\mathbf{Y}}
\newcommand{\bsigma}{\boldsymbol{\sigma}}

\title{Generalized couplings and ergodic rates for SPDEs and other  Markov models}

\date{26 March 2019}

\author[1,3,$*$]{Oleg Butkovsky}
\author[2,$**$]{Alexei Kulik}
\author[1,$**${}$*$]{Michael Scheutzow}
\affil[1]{\small {Technische Universit\"at Berlin,

Institut f\"ur Mathematik, MA 7-5, Fakult\"at II,

Strasse des 17.~Juni 136, 10623 Berlin, FRG. \bigskip
}}

\affil[2]{\small{Wroclaw University of Science and Technology,

Faculty of Pure and Applied Mathematics,

Wybrze\'ze Wyspia\'nskiego Str. 27, 50-370 Wroclaw, Poland.\bigskip
}}

\affil[3]{\small{Technion --- Israel Institute of Technology,

Faculty of Industrial Engineering and Management

Haifa, 3200003, Israel.\bigskip
}}

\begin{document}
\maketitle
\renewcommand{\thefootnote}{*}
\footnotetext{Email: \texttt{oleg.butkovskiy@gmail.com}. Supported in part by DFG Research Unit FOR 2402.}

\renewcommand{\thefootnote}{**}
\footnotetext{Email: \texttt{kulik.alex.m@gmail.com}.   }
\renewcommand{\thefootnote}{***}
\footnotetext{Email: \texttt{ms@math.tu-berlin.de}. Supported in part by DFG Research Unit FOR 2402.}

\begin{abstract}
We establish verifiable general  sufficient conditions for exponential or subexponential ergodicity of Markov processes that may lack the strong Feller property. We apply the obtained results to show exponential ergodicity of a variety of nonlinear stochastic partial differential equations with additive forcing, including 2D stochastic Navier--Stokes equations. Our main tool is a new version of the generalized coupling method.

\end{abstract}

\section{Introduction}

There
are many Markov processes that are not strong Feller and whose transition probabilities are mutually singular. Their ergodic properties cannot be analyzed with classical methods from, e.g., \cite{MT} and require  special treatment. One of the first papers in this direction was \cite{HMS11}; the results obtained there have been later generalized and extended in \cite{Bu14}, \cite{DFM}, \cite[Chapter~4]{K17}. These works provide sufficient conditions for exponential or subexponential convergence of transition probabilities of a Markov process towards its invariant measure in the Wasserstein metric.

Unfortunately,  checking these conditions in practice turns out to be rather difficult. One of the main results of this paper (Theorem~\ref{T:Lyap})   provides a set of \textit{verifiable} sufficient conditions for exponential or subexponential ergodicity. Furthermore, we develop a special framework to ease the application of these conditions to stochastic partial differential equations (SPDEs). To illustrate the applicability of our framework we establish exponential ergodicity of {\blue five} important nonlinear SPDE models, including the 2D stochastic Navier--Stokes equation. This generalizes the corresponding theorems from \cite{GMR17} and \cite[Section~6]{KS17}. To obtain these results we develop a new version of the \textit{generalized coupling} method, extending the ideas of \cite{Mat2002} and \cite{H02}.

\medskip
Let us recall that the classical approach to the analysis of ergodic properties of Markov processes (\cite{Numellin}, \cite{MT}) combines a  local mixing  condition on a certain set (a \emph{small set}) with a recurrence condition that the Markov process visits this set often enough.
 If these conditions are satisfied, then the Markov process has a unique invariant probability measure (IPM) and the transition probabilities converge to the IPM in total variation; the rate of convergence is determined by the recurrence rate, i.e., how quickly the Markov process returns to the small set. This can be quantified in terms of a Lyapunov function or the moments of the first return times.

The classical approach  works quite well for strong Feller Markov processes such as stochastic differential equations with uniformly elliptic noise, MCMC in finite dimensions, non-linear autoregressive models, etc.
If a Markov process lacks the strong Feller property, then usually the local mixing condition does not hold. Possible examples of processes with ``bad'' mixing properties include many infinite-dimensional Markov processes such as SPDEs, stochastic functional differential equations (SFDEs), MCMC in infinite dimensional Hilbert spaces, Markov processes with switching, etc. Typically the transition probabilities of these processes  with different initial points are  mutually singular, and thus ergodicity in total variation clearly fails.   For just one example of such a system, we refer to \cite{Scheu}, where the \emph{intrinsic memory} effect was observed for an SFDE: once the value of the segment process is known at some time moment $t$, the values at all the previous time moments are uniquely and measurably defined; see also the discussions in \cite{HMS11} and \cite[Section 4.1]{K17}. Clearly, the intrinsic memory effect yields mutual singularity of transition probabilities for the Markov process; in addition, such an effect prevents the mixing property of the system, which is a prerequisite for the classical approach.

To overcome this difficulty and to analyze the ergodic behaviour of Markov processes without the strong Feller property, we adopt the method developed by  M.~Hairer, J.~Mattingly and M.~Scheutzow in \cite{HMS11}. The method is  based on the concept of a \textit{$d$--small set}, which is a much more general object than a small set, see Section~\ref{S:1} below for the precise definition. The strategy in a sense mimics the classical one; namely, it shows that  if a Markov process visits a $d$-small set often enough, then, under some additional constraint on a Markov kernel (the so-called \textit{non-expansion property}), there exists  a unique IPM  and the transition probabilities weakly converge to the IPM  with a rate that is quantified by means of a properly chosen probability metric; again the convergence rate is determined by the recurrence properties of the chain. For  further extensions and generalizations within this approach, we refer to \cite{Bu14}, \cite{DFM}, \cite[Chapter 4]{K17}.

The crucial step in the above approach is to construct a \emph{distance-like} function $d$ such that the Markov kernel indeed has the non-expansion property w.r.t. $d$ and that a certain set is indeed $d$--small. This step can be quite challenging; we refer to \cite[Section~5.2]{HMS11}, where such a construction is provided for an SFDE, though in a complicated and nontransparent way, which makes  further extensions and modifications very difficult. One of the main results of our paper, Theorem \ref{L:contr}, provides a transparent and fairly simple algorithm for the construction of a distance-like function $d$ with required properties. We strongly believe this to be a substantial complement to the approach initiated in  \cite{HMS11} and developed in \cite{Bu14}, \cite{DFM}, \cite{K17}, which hopefully makes this method well applicable for a wide variety of Markov systems.

The latter assertion is supported by several particular Markov models, analyzed in the second part of the paper.
First, as a test case for our approach we reconsider SFDEs. We obtain a simple and transparent proof of their weak ergodicity at exponential, sub-exponential, and polynomial rates. Second, we consider five particular nonlinear SPDE models where the analysis of the ergodic properties is substantially more difficult due to the presence of strongly singular terms, such as the non-linear gradient term $(\u\cdot  \nabla)\u$ in the  Navier--Stokes equation. Using the general algorithm developed in the first part of the paper, we manage to overcome these difficulties, and establish exponential ergodicity in each of these five models. This in particular extends  the results from \cite{GMR17} and \cite[Section~6]{KS17}, where  unique ergodicity and  weak convergence to the IPM without a specified rate are proved for the same models.

Our approach is based on the concept of a \emph{generalized coupling}, which we briefly explain here. By a \emph{coupling} one usually means a pair of stochastic processes  with prescribed laws of the components; that is, the marginal laws.  Heuristically, by a \emph{generalized coupling} we will mean a pair of stochastic processes, whose marginal laws are not necessarily equal to a prescribed pair of probability distributions, but are in a sense close to this pair. Under the name \emph{asymptotic coupling}, this concept was introduced in  \cite{HMS11} in the spirit of earlier works \cite{H02}, \cite{Mat2002}, \cite{Mat2003}, and was later developed in \cite{GMR17}. Originally, the  marginal laws of the asymptotic coupling were assumed to be absolutely continuous with respect to the prescribed probability distributions; this, in a combination with the Birkhoff theorem, appears to be a convenient tool for proving  unique ergodicity. Moreover, it was shown in \cite{KS17} that the same notion can be used to guarantee weak stabilization of the transition probabilities; since this notion was used in a non-asymptotic way, the name was changed there to a \emph{generalized coupling}.

In the paper we further develop the above ideas and introduce a new definition of a generalized coupling.  We require   certain non-asymptotic total variation bounds  between the marginal distributions and the prescribed laws to hold true; see Section \ref{S:MR}. We show that the existence of  a ``good'' generalized coupling makes it easy to construct a non-expanding distance-like function $d$ with a large class of $d$-small sets; this, combined  with  recurrent--type conditions, provides ergodic rates. Thus there are two important advantages of our method compared to previous techniques. On the one hand,  generalized couplings with the required properties can be constructed quite efficiently using stochastic-control-type arguments, which makes the entire approach quite flexible and easy to apply.
On the other  hand, unlike  previous results based on  asymptotic/generalized couplings, our method provides  explicit bounds on convergence rates, rather than just uniqueness of the IPM or weak convergence of transition probabilities. Let us mention briefly, that the ergodic rates for a Markov process have further natural applications to limit theorems for additive functionals of the process and other, more complicated systems with Markov perturbations. In this context, there is a clearly seen difference between the stabilization of transition probabilities (with non-specified rate) on the one hand,  and explicit ergodic rates on the other hand. The first one typically guarantees the Law of Large Numbers (or more generally, \emph{averaging principle}), while the second one gives rise to the Central Limit Theorem and \emph{diffusion approximation}-type theorems; see \cite[Section 5.4, comment 2 and Remark 6.3.9]{K17}. For further discussion, references, and systematic treatment of the limit theorems for systems with weakly ergodic Markov perturbations we refer to \cite[Chapters 5,6]{K17}.

The rest of the paper is organized as follows. In Section 2.1 we recall well-known results  concerning the convergence of Markov processes. We present our main results in Section 2.2. In Section 3 we give an application to SFDEs. In Section 4 we develop a general framework for proving exponential ergodicity for nonlinear SPDEs and establish exponential ergodicity of a number of important SPDEs.

\bigskip

\noindent \textbf{Acknowledgments}. The authors are grateful to Jonathan Mattingly for helpful discussions, to
Mikhail Lifshits for fruitful comments, {\blue and to Alexander Shaposhnikov for bringing the article \cite{Us} to our attention. The authors also would like to thank the anonymous referee for the valuable remarks  and
suggestions which helped to improve the quality of the paper.}
Part of the work on the project has been done during the visit of OB to the Fields Institute
(Toronto, Canada) in October~2017. OB is very grateful to the Fields institute and especially to Bryan Eelhart for their support, hospitality, and incredible coffee breaks. The work on the project has been finished  during the visit of AK to the Technical University of Dresden (Germany); AK is very grateful to to the Technical University of Dresden and especially to Ren\'e Schilling for their support and hospitality.

\section{A general setup for ergodicity and estimating convergence rates}\label{S:1}

This section consists of two parts. In the first part we introduce the basic notation and recall some useful definitions and important theorems. In the second part  we present our main results.

\subsection{Generalities and notation}\label{S:11}

Let $(E,\rho)$ be a Polish space and $\mathcal{E}=\B(E)$ be the corresponding Borel $\sigma$-field. Consider a Markov transition function $\{P_t (x, A), x\in E, A\in \mathcal{E}\}_{t\in \R_+}$. We use the standard notation for the corresponding semigroup of integral operators
\begin{equation*}
P_t \phi(x)=\int_E \phi(y) P_t(x,dy),\quad x\in E,\, t\in\R_+.
\end{equation*}
We also denote by $\{\P_x, x\in E\}$ the corresponding Markov family, i.e.,  $\P_x$ denotes the law of the Markov process $X=\{X_t, t\geq 0\}$ with $X_0=x$ and the given transition function. The law of $X$ will be understood in the sense of finite-dimensional distributions;  that is, we will not rely on the trajectory-wise properties of $X$.

A function $d:E\times E\to\R_+$ is called \emph{distance-like} if it is symmetric, lower semicontinuous, and $d(x,y)=0\Leftrightarrow x=y$; see \cite[Definition~4.3]{HMS11}.
Denote by $\mathcal{P}(E)$ the set of all probability measures on $(E,\mathcal{E})$, and  for $\mu,\nu\in  \mathcal{P}(E)$ denote by $\C(\mu,\nu)$  the set of all \textit{couplings} between $\mu$ and $\nu$, i.e.  probability measures on $(E\times E,\mathcal{E}\otimes\mathcal{E})$ with marginals $\mu$ and $\nu$. For a given distance-like function $d$, the corresponding \emph{coupling distance}  $W_d:\mathcal{P}(E)\times\mathcal{P}(E)\to\R_+\cup\{\infty\}$ is defined by
\begin{equation*}
W_d(\mu,\nu):=\inf_{\lambda\in\C(\mu,\nu)}\int_{E\times E} d(x,y) \lambda(dx,dy),\quad \mu,\nu\in\mathcal{P}(E).
\end{equation*}
  If $d$ is a lower semicontinuous metric on $E$, then $W_d$ is the usual \textit{Wasserstein-$1$} (or \textit{Kantorovich}) distance. In particular, if $d$ is the discrete metric, i.e., $d(x,y)=\I(x\neq y)$, then $W_d$ coincides with
the \textit{total variation} distance $d_{TV}$; the latter can also be defined
as follows:
\begin{equation*}
d_{TV}(\mu,\nu):=\sup_{A\in\mathcal{E}}|\mu(A)-\nu(A)|.
\end{equation*}

Recall that if the original metric $\rho$ is  bounded, then convergence in total variation implies convergence in $W_{\rho}$; furthermore in this case convergence in $W_{\rho}$ is equivalent to   weak convergence.

As explained in the introduction, if a Markov process lacks the strong Feller property, then the total variation metric is too strong to measure the distance between the transition probabilities of the process and its invariant measure; this is the reason for us to focus on the study of weaker distances $W_d$. As mentioned above, \cite{HMS11} and subsequent works \cite{Bu14}, \cite{DFM}, \cite[Chapter~4]{K17} provide sufficient  conditions for convergence of the law of a Markov process to its invariant measure in $W_d$ for properly chosen $d$. For the convenience of the reader we formulate here two results of such  type, which actually  motivate the form in which our main results are presented below. For additional details of the theory, discussions, and examples we refer to the above--mentioned references.

Define for a
measurable function $\phi\colon\R_+\to(0,\infty)$
\begin{equation}\label{Hphi}
H_\phi(x):=\int_1^x \frac{1}{\phi(u)}\,du,\quad x\ge1.
\end{equation}

\begin{Proposition}[{\cite[Theorem~2.4]{Bu14}, \cite[Theorem~4.5.2]{K17}}]\label{Prop:21}
Assume that the Markov semigroup $P$ is Feller. Suppose there exist a measurable function $V\colon E\to[0,\infty)$ and a bounded distance--like function $d$ on $E$ such that the following conditions hold:
\begin{itemize}
 \item[{\rm1.}] $V$ satisfies the Lyapunov condition, that is, there exist a concave differentiable function
 $\phi\colon\R_+\to\R_+$ increasing to infinity with $\phi(0)=0$ and a constant $K>0$ such that
 for any $t\ge0$, $x\in E$
 \begin{equation}\label{Lyapf}
  P_tV(x)\le V(x)-\int_0^t P_s (\phi\circ V)(x)\,ds + K t.
 \end{equation}
\item[{\rm2.}] One has $\rho\le d$, where $\rho$ is the original metric on the Polish space $E$.

\item[{\rm3.}] There exist $0<t_1<t_2<\infty$ such that for all $t\in[t_1,t_2]$
\begin{equation}\label{nonexpanding}
 W_d(P_t(x,\cdot),P_t(y,\cdot))\le d (x,y),\quad x,y\in E
\end{equation}

\item[{\rm4.}] There exists $t>0$ such that for any $M>0$ there exists $\eps=\eps(M,t)>0$ such that
\begin{equation}\label{dsmall}
 W_d(P_{t}(x,\cdot),P_{t}(y,\cdot))\le (1-\eps)d (x,y),\quad x,y\in \{V\le M\}
\end{equation}
\end{itemize}

Then the Markov semigroup $(P_t)$ has a unique invariant measure $\pi$. Moreover, for any  $\delta\in(0,1)$ there exist constants $C_1,C_2>0$ such that for any $x\in E$,
\begin{equation*}
W_{d} (P_t(x,\cdot),\pi)\le\frac{C_1 (1+\phi(V(x))^\delta)}{\phi(H_\phi^{-1}(C_2 t))^{\delta}},\quad t\ge0.
\end{equation*}
\end{Proposition}

If the function $V$ satisfies a stronger version of the Lyapunov inequality, then condition~4 of the above proposition can be relaxed. More precisely, the following statement holds.

\begin{Proposition}[{\cite[Theorem~4.8]{HMS11}}]\label{Prop:22} Assume that condition~1 of Proposition~\ref{Prop:21} holds with  $\phi(x)=\gamma x$, $\gamma>0$. Then condition~4  of Proposition~\ref{Prop:21} can be replaced with the following weaker condition:
\begin{itemize}
\item[{\rm4*.}]\! There exist $\eps\!>\!0$, $t\!>\!0$ such that inequality \eqref{dsmall} is satisfied for
all  $x,y\!\in\!\{V\!\le\! 4K/\gamma\}$.
\end{itemize}
\end{Proposition}

Thus, the difference between the assumptions in the exponential and the subexponential cases is the following. In the exponential case it is enough to show that for any $M>0$ there exists $t>0$ such that the level set $\{V\le M\}$ satisfies inequality \eqref{dsmall}. On the other hand,  in the subexponential case one cannot pick $t$ depending on $M$; one has to check inequality \eqref{dsmall} for some fixed time $t$ simultaneously for \textit{all} $M$. We will see later in Section~\ref{S:MR} that this difference will become crucial and we will present different sets of assumptions for the subexponential and exponential cases.

A convenient tool to verify  \eqref{nonexpanding} and \eqref{dsmall} is provided by the following pair of notions.
\begin{Definition}[{\cite[Definition~4.6]{HMS11}}]\label{D:21}
A distance-like function $d$ bounded by $1$ is called \textit{contracting}
for $P_t$ if there exists  $\alpha<1$ such that for any $x,y\in E$ with $d(x,y)<1$
we have
\begin{equation*}
W_d(P_t(x,\cdot),P_t(y,\cdot))\le \alpha d(x,y).
\end{equation*}
\end{Definition}

\begin{Definition}[{\cite[Definition~4.4]{HMS11}}]\label{D:22}
A set $B\subset E$ is called \textit{$d$--small} for $P_t$ if for some $\eps>0$
\begin{equation*}
\sup_{x,y\in B}W_d(P_t(x,\cdot),P_t(y,\cdot))\le 1-\eps.
\end{equation*}
\end{Definition}


Clearly, if $d$ is contracting for $P_t$ and the set  $\{V\le M\}$ is $d$--small for $P_t$, then inequalities  \eqref{nonexpanding} and \eqref{dsmall} are satisfied.

Let us briefly mention a closely related concept of a \textit{coarse Ricci curvature} \cite[Definition~3]{Oll},   introduced  recently in  somewhat different framework, mainly focused at functional inequalities. In the terms of \cite{Oll},  conditions 3 and 4 of Proposition~\ref{Prop:21} can be reformulated as follows: the Ricci curvature of the whole space should be non-negative and the Ricci curvature of any level set of $V$ should be positive.




\subsection{Main results}\label{S:MR}

Propositions~\ref{Prop:21} and \ref{Prop:22} allow one to bound the rate of convergence of a Markov process to its invariant measure in the Wasserstein metric. The crucial question within this context is how to choose the distance-like function $d$ in a way that conditions 3 and 4 (or 4*) are satisfied. The main goal of this section is to provide an algorithm of such construction, which relies on explicit and easy-to-verify set of assumptions.

We will call a function $\theta:E\times E\to\R_+$ a \textit{premetric} if it is lower semicontinuous and $\theta(x,y)=0\Leftrightarrow x=y$; we fix a premetric $\theta$ on $E$ till the end of this section. Our first assumption will serve as a replacement for the contractivity condition.

\begin{Assumption}{A}\label{A:1} There exist a  non-increasing function
$r:\R_+\to\R_+$ with $\lim_{t\to\infty} r(t)=0$ and a {\blue locally bounded} function $L:\R_+\to\R_+$ such that for any $x,y\in E$, there exist random processes $X^{x,y}=(X^{x,y}_t)_{t\ge0}$, $Y^{x,y}=(Y^{x,y}_t)_{t\ge0}$ with the following properties:
\begin{enumerate}
   \item[{\rm1.}] $\mathrm{Law}\, (X^{x,y})=\P_x,$ and
     \begin{equation*}
     d_{TV}(\Law(Y^{x,y}_t), P_t(y,\cdot))\leq L(t)\theta(x,y),\quad t\ge0;
     \end{equation*}
   \item[{\rm2.}]$\E\theta(X^{x,y}_t, Y^{x,y}_t)\leq r(t) \theta(x,y)$, $t\ge0$.
\end{enumerate}
\end{Assumption}

Assumption \textbf{\ref{A:1}}.1 means that the pair  $(X^{x,y}, Y^{x,y})$ has the sense of a generalized coupling: its first component is distributed as $\P_x$ and the defect between the law of the second component at time $t$ and the ``true law'' $P_t(y,\cdot)$ can be effectively controlled. The function $L$ here plays the role of a Lipschitz coefficient. We do not suppose that $L$ is decreasing to $0$ at infinity; moreover we allow the situation when $L(t)\to\infty$ as $t\to\infty$, that is, the defect between the generalized coupling and the ``true law'' might be quite large.  Assumption {\bf\ref{A:1}}.2 controls the decay rate of the premetric $\theta$ between the components of this generalized coupling. Note also that we are not imposing any additional conditions on the relation between $\P_y$ and $\Law(Y^{x,y})$; in particular we are not assuming any absolute continuity.

The second assumption will replace the $d$-small property.
\begin{Assumption}{B1}\label{A:22}
There exist a set $B\subset E$ and $t_0>0$ such that for any $\eps>0$ there exists a set  $D\in\mathcal{E}$ such that
 \begin{enumerate}
   \item[{\rm1.}] $\inf\limits_{x\in B}P_{t_0}(x,D)>0$;
\item[{\rm2.}] $\sup\limits_{x,y \in D}\theta(x,y)\le \eps$.
 \end{enumerate}
\end{Assumption}
We note that a similar in spirit assumption can be found in the literature, see, e.g., \cite[Corollary 2.2]{HM11}.
Assumption \textbf{\ref{A:22}} is not of  a ``generalized-coupling-type''; it is essentially a support condition. However in the exponential case due to the reasons discussed after Proposition~\ref{Prop:22}, it can be replaced by  a ``generalized-coupling-type'' assumption that may be easier to check in some setups.

\begin{Assumption}{B2}\label{A:2} There exist a set $B\subset E$, a  non-increasing function $R:\R_+\to\R_+$ with $\lim_{t\to\infty} R(t)=0$, and $\eps>0$ such that for any $x,y\in B$, there exist random processes  $X^{x,y}=(X^{x,y}_t)_{t\ge0}$, $Y^{x,y}=(Y^{x,y}_t)_{t\ge0}$ with the following properties:
 \begin{enumerate}
   \item[{\rm1.}] $\mathrm{Law}\, (X^{x,y})=\P_x,$ and
\begin{equation}\label{TV_B}
d_{TV}(\Law(Y^{x,y}_t), P_t(y,\cdot))\leq 1-\eps,\quad t\ge0,\,\, x,y\in B;
\end{equation}
\item[{\rm2.}]$\E\theta(X^{x,y}_t, Y^{x,y}_t)\leq R(t)$, $t\ge0$, $x,y\in B$.
 \end{enumerate}
\end{Assumption}

Assumption \textbf{\ref{A:2}} means that there exists a ``good'' set $B\subset E$, where we have a family of generalized couplings with especially nice properties. Namely, the first component of the generalized coupling coincides with the ``true law'', and the distance between the second component at time $t$ and the corresponding ``true law'' at time $t$  is uniformly bounded on~$B$.

Now  for any fixed $N>0$ we consider the following distance-like function:
\begin{equation}\label{dn}
d_{N}(x,y):=N\theta(x,y)\wedge N\theta(y,x)\wedge 1,\quad x,y\in E.
\end{equation}
Our first main result shows how the $d$-smallness and contractivity conditions can be verified under the assumptions \textbf{\ref{A:1}}, \textbf{\ref{A:22}}, and \textbf{\ref{A:2}}.

\begin{Theorem}\label{L:contr}
\begin{enumerate}
\item[\rm{(i)}] Assume that {\rm\textbf{\ref{A:1}}} holds for some functions $r$, $L$. Then for any $N>0$, $t>0$ such that $r(t)\le 1/3$ and $N\ge 2 L(t)$
the distance--like function $d_N$ is contracting for $P_t$ {\blue (in the sense of Definition~\ref{D:21})}.
\item[\rm{(ii)}] Assume that  {\rm\textbf{\ref{A:22}}} holds for some $t_0>0$ and a set $B\subset E$. Then for any $N>0$  the set $B$ is $d_N$--small for $P_{t_0}$ {\blue (in the sense of Definition~\ref{D:22})}.
\item[\rm{(iii)}] Assume that {\rm\textbf{\ref{A:2}}} holds for some function $R$, a set $B\subset E$ and $\eps>0$. Then for any $N>0$, $t>0$ such that $N R(t)\le \eps/2$ the set $B$ is $d_N$--small for $P_t$.
\end{enumerate}
\end{Theorem}
\begin{proof}
(i). Fix $t>0$, $N>0$ that satisfy the assumptions of the theorem. Take any $x,y\in E$ such that $d_N(x,y)<1$. Without loss of generality, suppose that $\theta(x,y)\le \theta(y,x)$.

We begin by observing that by Assumption {\bf\ref{A:1}}.1 and the coupling lemma (see, e.g., \cite[Theorem~4.1]{Villani}),
there exists a pair of random variables $\wh Y,Z$ such that
$\Law(\wh Y)=\Law(Y^{x,y}_t)$, $\Law(Z)=P_t(y,\cdot)$ and
\begin{equation*}
\P(\wh Y\neq Z)=d_{TV}(\Law(Y^{x,y}_t), P_t(y,\cdot))\le L(t) \theta(x,y).
\end{equation*}

By the gluing lemma (see, e.g., \cite[Lemma~4.3.2]{K17} and \cite[p.~23]{Villani}), we can assume that $\wh Y, Z$ are defined on the same probability space with $X^{x,y}_t, Y^{x,y}_t$, and $\wh Y=Y^{x,y}_t$. Then the joint law of $X^{x,y}_t, Z$ is a  coupling for {\blue $P_t(x,\cdot)$, $P_t(y,\cdot)$}, and thus
\begin{align*}
W_{d_N}(P_t(x, \cdot), P_t(y, \cdot))\le&  \E {d_N}(X_t^{x,y}, Z)\\
=&  \E {d_N}(X_t^{x,y}, Z)\I(Y^{x,y}_t=  Z)+\E {d_N}(X_t^{x,y}, Z)\I(Y^{x,y}_t\not=  Z)\\
\le& \E  {d_N}(X_t^{x,y}, Y_t^{x,y})+\P( Y^{x,y}_t\not=  Z)\nn\\
\le& N\E  \theta(X_t^{x,y}, Y_t^{x,y})+ L(t) \theta(x,y)\nn\\
\le& N\theta(x,y)(r(t)+L(t)N^{-1})\\
\le& \frac56 N\theta(x,y)=\frac56 d_N(x,y),\nn
\end{align*}
where the last equality follows from the assumptions $d_N(x,y)<1$ and $\theta(x,y)\le \theta(y,x)$. This proves the contraction property.

\smallskip

(ii). Fix $N>0$. In this case we show  the $d_N$-smallness of $P_{t_0}$ using the independent coupling, similarly to \cite[Proposition~5.3]{HMS11}. Take $\eps:=1/(2N)$. By Assumption~{\textbf{\ref{A:22}}} there exists a set $D\in\mathcal{E}$ such that $\delta:=\inf_{x\in B}P_{\blue t_0}(x,D)>0$ and $\sup_{x,y \in D}\theta(x,y)\le \eps$.

Now  we fix $x,y\in B$ and construct independent random variables $X$, $Y$ such that
\begin{equation*}
\Law(X)=P_{t_0}(x,\cdot);\,\, \Law(Y)=P_{t_0}(y,\cdot).
\end{equation*}
Then
\begin{align*}
W_{d_N}(P_{t_0}(x, \cdot), P_{t_0}(y, \cdot))\le&  \E {d_N}(X, Y)\le \P(\{X\notin D\} \cup \{Y\notin D\})+N\eps
\P(X\in D, Y\in D)\\
\le& 1- \delta^2 (1-N\eps)\le 1-\frac 12 \delta^2,
\end{align*}
 showing that $B$ is $d_N$-small.

(iii). We establish  the $d_N$-smallness of $P_t$ by an argument, similar to the one used in the proof of Part (i) of the lemma. We fix $t>0$, $N>0$, $x,y\in B$. Using the same notation and the same construction as in the first part of the proof,
we obtain a triple of random variables  $X_t^{x,y}, Y_t^{x,y}, Z$ such that
\begin{align*}
&\Law(X_t^{x,y})=P_t(x,\cdot);\quad \Law(Z)=P_t(y,\cdot);\\
&\E\theta(X^{x,y}_t, Y^{x,y}_t)\leq R(t);\quad \P\Big(Y^{x,y}_t \neq Z \Big)\le 1-\eps.
\end{align*}
Hence, using the fact that $d_N\le N\theta$, we deduce
\begin{align*}
W_{d_N}(P_t(x, \cdot), P_t(y, \cdot))\le&  \E {d_N}(X_t^{x,y}, Z)\le \E  {d_N}(X_t^{x,y}, Y_t^{x,y})+
\P( Y^{x,y}_t\not=  Z)\\
\le& N R(t)+1-\eps\\
\le& \eps/2+1-\eps=1-\eps/2,
\end{align*}
where the fourth inequality follows from the fact that $NR(t)\le \eps/2$. This proves the $d_N$-smallness of $B$.
\end{proof}
\begin{Remark}
 We
  will see in Sections \ref{S:SFDE} and  \ref{S:SPDE} that in various cases  generalized couplings with the  required properties can be constructed using a stochastic control approach. Then our  entire argument will  consist of two principal steps:
\begin{itemize}
  \item[(i)] using a  stochastic control technique, we construct a generalized coupling $X^{x,y},Y^{x,y}$, which exhibits certain contraction properties (Assumptions {\bf\ref{A:1}}.2, {\bf\ref{A:2}}.2)
  \item[(ii)] using the error-in-law bounds for this generalized coupling (Assumptions {\bf\ref{A:1}}.1, {\bf\ref{A:2}}.1), we construct then a true coupling $X^{x,y}_t,Z$ for $P_t(x, \cdot), P_t(y, \cdot)$, which  confirms the required contraction property of $d_N$.
\end{itemize}

We will call this type of argument a \emph{Control-and-Reimburse} strategy. The same general idea --- to apply the stochastic control in order to improve the system, and then to
take into account the impact of the control --- was  actually used, in more implicit  form,  in the previously mentioned construction \cite[Section~5.2]{HMS11} of the distance-like function $d(x,y)$ for an SFDE, and in the approach to the study of weak ergodicity of SPDEs developed in \cite{H02}. A similar argument was used to prove ergodicity in total variation in \cite{AVer}    for degenerate diffusions, and in  \cite{BodKul08} for solutions to L\'evy driven SDEs. Related ideas were used to establish the Harnack inequality for SDEs and SFDEs \cite{Wang, ERS09}.

\end{Remark}

Our next result establishes unique ergodicity under Assumptions {\bf\ref{A:1}} and {\bf\ref{A:22}}. It is a combination of Theorem~\ref{L:contr} and Proposition~\ref{Prop:21}. {\blue Recall the definition of $H_\phi$ in \eqref{Hphi}}.

\begin{Theorem}\label{T:Lyap} {\blue Assume that the Markov semigroup $P$ is Feller. Suppose that
\begin{enumerate}[{\rm1.}]
\item There exists a measurable function $V$ that satisfies condition~1 of Proposition~\ref{Prop:21}.
\item $\rho\wedge1\le \theta$, where $\rho$ is the original metric on E.
\item Assumption {\bf\ref{A:1}}  holds for  functions $r$, $L$.
\item There exists $t_0>0$ such that for any $M>0$ Assumption {\bf\ref{A:22}} holds for $B=\{V\le M\}$ and $t_0$.
\end{enumerate}
}

Then the Markov semigroup $(P_t)$ has a unique invariant measure $\pi$. Moreover, for any $\delta\in(0,1)$ there exist constants $C_1,C_2>0$ such that for any $x\in E$,
\begin{equation}\label{finalres}
W_{\rho\wedge 1} (P_t(x,\cdot),\pi)\le \frac{C_1 (1+\phi(V(x))^\delta)}{\phi(H_\phi^{-1}(C_2 t))^{\delta}},\quad t\ge0.
\end{equation}
\end{Theorem}

\begin{proof}[Proof of Theorem~\ref{T:Lyap}]

Let us check that all the conditions of Proposition~\ref{Prop:21} are satisfied for the function $V$ and a distance--like function $d_N$ introduced in \eqref{dn}, where $N\ge1$ will be chosen later. The first condition of Proposition~\ref{Prop:21} is satisfied by the first  assumption of  the theorem. Since $\rho\wedge1\le \theta$, we have
 $ \rho\wedge 1\le  d_N$ and thus the second condition of Proposition~\ref{Prop:21} is satisfied.

Now we pick $t_*\ge{\blue 0}$ such that $r(t_*)\le1/3$. This is possible thanks to Assumption~{\bf\ref{A:1}}. Set {\blue$N:=(2\sup_{t\in[t_*,t_*+t_0]}L(t))\vee 1$.}
Theorem~\ref{L:contr}(i) implies that for some {\blue $\alpha<1$} we have
\begin{equation}\label{contralpha}
W_{d_N}(P_{t_*+t_0}(x,\cdot),P_{t_*+t_0}(y,\cdot))\le \alpha d_N(x,y)\quad\text{whenever {\blue $d_N(x,y)<1$}}
\end{equation}
and that the distance--like function $d_N$ is contracting for $P_t$ for any {\blue $t\in[t_*,t_*+t_0]$}. This immediately implies that inequality \eqref{nonexpanding} holds {\blue for any $t\in[t_*,t_*+t_0]$} and thus condition 3 of Proposition~\ref{Prop:21} is satisfied. {\blue In particular, we have
\begin{equation}\label{contralpha2}
W_{d_N}(P_{t_*}(x,\cdot),P_{t_*}(y,\cdot))\le d_N(x,y)\quad x,y\in E.
\end{equation}
}

To check condition 4 of Proposition~\ref{Prop:21} we take {\blue $t:=t_*+t_0$} and fix arbitrary $M>0$. It follows from Theorem~\ref{L:contr}(ii) and assumption~4 of the theorem that the set $\{V\le M\}$  is $d_N$--small for $P_{t_0}$ with some $\eps=\eps(t,M)$. Now take any $x,y\in \{V\le M\}$. If $d_N(x,y)<1$ then using \eqref{contralpha}, we get
\begin{equation*}
 W_{d_N}(P_{t}(x,\cdot),P_{t}(y,\cdot))\le  \alpha d_N (x,y).
\end{equation*}
If $d_N(x,y)=1$, then  using \eqref{contralpha2} and $d_N$--small property,  we derive
\begin{equation*}
 W_{d_N}(P_{t_0+t_*}(x,\cdot),P_{t_0+t_*}(y,\cdot))\le W_{d_N}(P_{t_0}(x,\cdot),P_{t_0}(y,\cdot))\le 1-\eps=(1-\eps) d_N (x,y).
\end{equation*}
Thus, condition 4 of Proposition~\ref{Prop:21} is met. Therefore all conditions of Proposition~\ref{Prop:21} are satisfied and hence the Markov semigroup $(P_t)$ has a unique invariant measure $\pi$ and \eqref{finalres} holds.
\end{proof}

 \begin{Theorem}\label{T:Lyap2}
 In the exponential case (that is, when condition~1 of Proposition~\ref{Prop:21} holds for a linear function $\phi(x)=\gamma x$, $\gamma>0$)
condition~4 of Theorem~\ref{T:Lyap} can  be replaced by the following assumption:
 \begin{itemize}
  \item[\rm4*.]  Assumption {\bf\ref{A:2}} holds for {\blue some function $R$}, the level set of the Lyapunov function
 $\{V\le 4K/\gamma\}$, {\blue and some $\eps>0$}.
 \end{itemize}
 \end{Theorem}

 \begin{proof}[Proof of Theorem~\ref{T:Lyap2}]
 {\blue
 We use essentially the same line of argument as in the proof of Theorem~\ref{T:Lyap} with some modifications.
 Pick $t_*\ge{\blue 0}$ such that $r(t_*)\le1/3$ and set $N:=(2\sup_{t\in[t_*,2t_*]}L(t))\vee 1$. 
 We see again that conditions 1--3 of Proposition~\ref{Prop:21} hold for $V$ and $d_N$.   In particular we have for some  $\alpha<1$
\begin{equation}\label{contralpha22}
W_{d_N}(P_{t_*}(x,\cdot),P_{t_*}(y,\cdot))\le \alpha d_N(x,y)\quad\text{whenever $d_N(x,y)<1$.}
\end{equation} 
Furthermore, it follows from Theorem~\ref{L:contr}(i) that for any $t\in[t_*,2t_*]$ we have
 \begin{equation*}
W_{d_N}(P_{t}(x,\cdot),P_{t}(y,\cdot))\le d_N(x,y)\quad x,y\in E.
\end{equation*}
This implies that for any $t\ge t_*$ 
 \begin{equation}\label{masscontract}
W_{d_N}(P_{t}(x,\cdot),P_{t}(y,\cdot))\le d_N(x,y)\quad x,y\in E.
\end{equation}
Pick now $t_0\ge t_*$ such that $NR(t_0)\le\eps/2$. Then
condition 4* of the theorem,  the definition of $N$, and  Theorem~\ref{L:contr}(iii) imply now that the set $B:=\{V\le 4K/\gamma\}$ is $d_N$--small for $P_{t_0}$.

 Let us check now that condition 4* of Proposition~\ref{Prop:22} holds for $t=t_*+ t_0$. Take any $x,y\in \{V\le 4K/\gamma\}$. If $d_N(x,y)<1$ then taking into account \eqref{contralpha22}, \eqref{masscontract}, and the fact that $t_0\ge t_*$, we get
\begin{equation*}
 W_{d_N}(P_{t_0+t_*}(x,\cdot),P_{t_0+t_*}(y,\cdot))\le W_{d_N}(P_{t_*}(x,\cdot),P_{t_*}(y,\cdot))\le \alpha d_N (x,y).
\end{equation*}
If $d_N(x,y)=1$, then  using \eqref{masscontract} and $d_N$--small property,  we derive
\begin{equation*}
 W_{d_N}(P_{t_0+t_*}(x,\cdot),P_{t_0+t_*}(y,\cdot))\le W_{d_N}(P_{t_0}(x,\cdot),P_{t_0}(y,\cdot))\le 1-\eps=(1-\eps) d_N (x,y).
\end{equation*}
Thus, condition 4* of Proposition~\ref{Prop:22} is met.

Therefore all conditions of Proposition~\ref{Prop:22} are satisfied and thus the Markov semigroup $(P_t)$ has a unique invariant measure $\pi$ and \eqref{finalres} holds.
}
\end{proof}

\section{Easy application: SFDEs}\label{S:SFDE}

In this section we illustrate our method of generalized coupling by establishing a   rate of convergence to the invariant measure for solutions to SFDEs. We show how  Assumptions \textbf{\ref{A:1}} and \textbf{\ref{A:22}} can be immediately verified for these processes. This can be regarded as a ``warm--up'' before the next section, which deals with SPDEs. Related ideas will be applied there in a  more complicated setup and  additional challenges will arise.

Let us introduce the  model. Fix $n,m\in\N$ and $r>0$. Denote by $\mathcal{C}=\mathcal{C}([-r,0],\R^n)$ the space of continuous functions endowed with the supremum norm $\|\cdot\|$. For a matrix $M\in\R^{d\times m}$ we denote by $\norm{M}$ its Frobenius norm, that is, $\norm{M}:=\sqrt{\sum M_{ij}^2}$. For a real number $a$, let $a_+:=\max(a,0)$.

Consider the following SFDE:
\begin{align}\label{SFDE}
&d X^x(t)=f(X_t^x) dt +g (X_t^x) dW(t),\quad t\ge0\\
&X^x_0=x\nn,
\end{align}
where $f:\C\to\R^n$ and $g:\C\to\R^{n\times m}$ are measurable functions, $W$ is an $m$-dimensional Brownian motion, the initial condition $x\in\C$, and we use the standard notation $X_t(s):=X(t+s)$, $s\in[-r,0]$.

Suppose that the function $f$ is continuous and bounded on bounded subsets of~$\C$. Further, suppose that $f$ is one-sided Lipschitz and $g$ is Lipschitz, that is, there exists $C>0$ such that for any $x,y\in\C$
\begin{equation*}
\langle f(x)-f(y),x(0)-y(0)\rangle_++\norm{g(x)-g(y)}^2\le  C\|x-y\|^2.
\end{equation*}
It is known that under these assumptions equation \eqref{SFDE} has a unique strong solution~\cite{RS08}. Moreover this solution $X=(X_t)_{t\ge0}$ is a Feller Markov process with the state space $(\C,\mathcal{B}(\C))$. Denote by $(P_t)_{t\ge0}$ the corresponding semigroup.  Assume also  the uniform non-degeneracy condition:
\begin{equation}\label{nd}
\sup_{x\in\C} \norm{g^{-1}(x)}<\infty,
\end{equation}
where $g^{-1}(x)$ denotes a right inverse of the matrix $g(x)$.

The next theorem describes the convergence rate of  $\Law(X_t)$ to its invariant measure in the Wasserstein metric.

\begin{Theorem}\label{T:31} Assume that there exists a Lyapunov function $V\colon\C\to\R_+$ that satisfies condition~1 of Proposition~\ref{Prop:21}. Suppose that \textbf{either} (i) or (ii) holds.
\begin{itemize}
\item[\rm{(i)}]  $\lim_{\|x\|\to\infty} V(x)=+\infty$.

\item[\rm{(ii)}]
$\lim_{|x(0)|\to\infty} V(x)=+\infty$. Assume additionally that  the drift and diffusion of SFDE \eqref{SFDE} satisfy the following growth condition: there exists $C>0$ such that for any $x\in\C$
\begin{equation}\label{growth}
\langle f(x), x(0)\rangle_++\norm{g(x)}^2\le C +C|x(0)|^2.
\end{equation}
\end{itemize}
Then  SFDE \eqref{SFDE} has a unique invariant measure $\pi$. Further,  $\Law(X_t)$ converges to $\pi$ in the Wasserstein metric $W_{\|\cdot\|\wedge1}$ and the rate of convergence is given by \eqref{finalres}.
\end{Theorem}

This theorem is essentially not new; it is only a  mild generalization of \cite[Section~5]{HMS11} and  \cite[Theorem~3.2]{Bu14}. However,  using the generalized coupling method, we managed to drastically simplify  the  key ingredient of the proof, namely, the verification of the contraction property for the solution of an SFDE. Since this method is model insensitive, similar arguments allow us to establish contraction properties of solutions of SPDEs, see Section~\ref{S:SPDE}. On the other hand,  transferring the ideas used in \cite[Section~5]{HMS11}  from SFDEs to SPDEs or other Markov models seems to be rather difficult (we mention though a recent preprint \cite{China}, where SFDEs with infinite memory are analyzed in a way quite similar to \cite[Section~5]{HMS11}).

Construction of a suitable Lyapunov function $V$ that satisfies condition \eqref{Lyapf} for a specific SFDE is a completely independent task and is out of the scope of the paper; we refer here to \cite{Scheu84}, \cite[Remark~5.2]{HMS11}, \cite{BS}, \cite[Section~3.2]{Bu14}, \cite[Section~4.6.1]{K17} for possible methods of building $V$. Let us just briefly explain the difference between conditions (i) and (ii) of Theorem~\ref{T:31}. Condition (ii)  allows to consider a larger and more natural class of Lyapunov functions. For example the function $V(x):=|x(0)|^2$, $x\in\C$ satisfies condition (ii) but not condition (i) of the theorem. The price to pay is that the drift and the diffusion of the SFDE have to satisfy a certain growth condition.


The proof of Theorem~\ref{T:31} is based on the following key lemma.

\begin{Lemma}\label{T:SFDE}
\begin{itemize}
\item[\rm{(i)}] There exist $N_0>0$, $t_0>0$ such that for any $N\ge N_0$, $t\ge t_0$
the distance
$$
d_N(x,y):=N\|x-y\|\wedge1
$$
is contracting for $P_t$.

\item[\rm{(ii)}] For any $N>0$, $t\ge 2r$, $M>0$ the level set
$$
B_M:=\{x\in\C: \|x\|\le M\},
$$
is $d_N$-small for $P_t$.

\item[\rm{(iii)}] Assume additionally that the drift and diffusion satisfy the growth condition \eqref{growth}.
Then for any $N>0$, $t\ge 3r$, $M>0$
the level set
$$
H_M:=\{x\in\C: |x(0)|\le M\}.
$$
is $d_N$-small for $P_t$.
\end{itemize}
\end{Lemma}


\begin{proof}[Proof of Lemma~\ref{T:SFDE}] The proofs of all  three parts of the lemma are based on the verification of  Assumptions \textbf{\ref{A:1}} and \textbf{\ref{A:22}} for the Markov semigroup $(P_t)_{t\ge0}$ and applying then Theorem~\ref{L:contr}. In all the cases we take $E=\C$, $\rho(x,y)=\theta(x,y)=\|x-y\|$. It is clear that the space $(\C,\rho)$ is Polish and $\theta$ is a premetric.

(i).  Let us check that  $(P_t)_{t\ge0}$ satisfies Assumption \textbf{\ref{A:1}}.  Let $\lambda>0$ be a parameter to be chosen later. For each $x,y\in\C$ we consider the following generalized coupling. We put $X^{x,y}=X^x$ and let $Y^{x,y,\lambda}$ be the strong solution of the following equation:
\begin{align*}
d Y^{x,y,\lambda}(t)&=f(Y^{x,y,\lambda}_t) dt +g (Y^{x,y,\lambda}_t) dW^{x,y,\lambda}(t),\quad t\ge0\\
Y^{x,y,\lambda}_0&=y,
\end{align*}
where
\begin{equation*}
d W^{x,y,\lambda}(t):=\beta^{x,y,\lambda}(t)dt +dW(t),
\end{equation*}
and
\begin{equation*}
\beta^{x,y,\lambda}(t):=\lambda g ( Y^{x,y,\lambda}_t)^{-1} (X^x(t)-Y^{x,y,\lambda}(t)).
\end{equation*}
The existence and uniqueness of $Y^{x,y,\lambda}$ follows again from \cite{RS08}. By \cite[Lemma~3.6]{HMS11}, there exists some $\lambda_0>0$ that does not depend on $x$, $y$  such that
\begin{equation}\label{boundbeta}
\E\|X_t^{x}-Y_t^{x,y,\lambda_0}\|^2\le C e^{- 2t}\|x-y\|^2,\quad t\ge0.
\end{equation}
From now on we take $\lambda=\lambda_0$ and denote $Y^{x,y}:=Y^{x,y,\lambda_0}$, $\beta:=\beta^{x,y,\lambda_0}$. By construction, $\Law(X^{x,y})=\P_x$. By Theorem~\ref{tKLGirsanov} and inequality  \eqref{otzenka2}, we have for any fixed $t\ge0$
\begin{align*}
d_{TV}\bigl(\Law(Y^{x,y}_t),P_t(y,\cdot)\bigr) &\le d_{TV}\bigl(\Law(W^{x,y}(s),\, s\in[0,t]),\Law(W(s),\, s\in[0,t])\bigr)\\
&\le \Bigl(\E \int_0^t |\beta(s)|^2\,ds\Bigr)^{1/2}\le {\blue C_1}\lambda_0\|x-y\|,
\end{align*}
{\blue for some $C_1>0$}, where in the last step we have used \eqref{nd} and \eqref{boundbeta}. Thus condition  \textbf{\ref{A:1}}.1 is satisfied with $L(t)={\blue C_1}\lambda_0$.
Using again \eqref{boundbeta}, we derive
\begin{equation*}
\E \|X_t^{x,y}-Y_t^{x,y}\|\le Ce^{- t}\|x-y\|.
\end{equation*}
Thus, condition \textbf{\ref{A:1}}.2 is satisfied with $r(t)={\blue C_2}e^{-t}$,
{\blue for some $C_2>0$}.

Take now $N_0:={\blue 2C_1}\lambda_0$ and choose $t_0$ such that ${C_2}e^{-t_0}\le 1/3$. Then by above the strong solution of SFDE \eqref{SFDE} satisfies all the conditions of Theorem~\ref{L:contr}(i) for any $N\ge N_0$ and $t\ge t_0$. Thus the distance $d_N$ is contracting for $P_t$ for any  $N\ge N_0$ and $t\ge t_0$.

(ii). Fix $t\ge 2r$ and $M>0$. Let us check that Assumption \textbf{\ref{A:22}} holds for $B_M$. Given $\eps>0$ we take $D:=\{x\in \C\colon \|x\|\le\eps/2\}$. Then, clearly,  \textbf{\ref{A:22}}.2 holds. By \cite[Lemma~3.8]{HMS11} \textbf{\ref{A:22}}.1 is also satisfied. Therefore, the statement of the lemma follows from Theorem~\ref{L:contr}(ii).

%

(iii). Fix ${\blue t_0}\ge3r$ and $M>0$. As in the proof of part (ii) of the lemma, let us verify that Assumption \textbf{\ref{A:22}} holds for  $H_M$. Given $\eps>0$ we set  again $D:=\{x\in \C\colon \|x\|\le\eps/2\}$.  Clearly,  \textbf{\ref{A:22}}.2 holds.

To check \textbf{\ref{A:22}}.1 we note first that by the It\^o formula we have for any $t\ge0$
\begin{equation*}
d|X^{x}(t)|^2=2\langle X^{x}(t),f(X^{x}_t)\rangle dt+\norm{g(X^{x}_t)}^2dt+dM(t),
\end{equation*}
where $M$ is a local martingale with $M(0)=0$ and $dM(t)=2\langle X^{x}(t), g(X^{x}_t)dW(t)\rangle$. {\blue For arbitrary $\delta>0$ } let $\tau_{\blue \delta}:=\inf\{t\ge0:\,|X^{x}(t)|\ge{\blue \delta}^{-1}\}$. Applying
the Burkholder--Davis--Gundy inequality together with assumption \eqref{growth}, we derive for any $t\in[0,r]$
\begin{align*}
\E\sup_{s\in[0,t\wedge\tau_{\blue \delta}]}|X^{x}(s)|^4&\le C|x(0)|^4+ C\E\int_0^{t\wedge\tau_{\blue \delta}} (1+|X^{x}(s)|^4)\,ds\\
&\le C|x(0)|^4+ C\E\int_0^{t\wedge\tau_{\blue \delta}} (1+\sup_{u\in[0,s\wedge\tau_{\blue \delta}]}|X^{x}(u)|^4)\,ds\\
&\le C|x(0)|^4+ C\E\int_0^{t} (1+\sup_{u\in[0,s\wedge\tau_{\blue \delta}]}|X^{x}(u)|^4)\,ds,
\end{align*}
where the constant $C>0$ does not depend on ${\blue \delta}$. Clearly,
$\E\sup_{s\in[0,t\wedge\tau_{\blue \delta}]}|X^{x}(s)|^4\le {\blue \delta}^{-4}<\infty$ for any $t\in[0,1]$. Therefore, the Gronwall inequality yields
\begin{equation*}
\E\sup_{s\in[0,r\wedge\tau_{\blue \delta}]}|X^{x}(s)|^4\le C|x(0)|^4+C,
\end{equation*}
where the constant $C>0$ is again independent of ${\blue \delta}$. By Fatou's lemma we finally obtain
\begin{equation*}
\E\sup_{s\in[0,r]}|X^{x}(s)|^4\le C|x(0)|^4+C.
\end{equation*}
Thus, for some large $L=L(M)>0$, by the Chebyshev inequality we have
\begin{equation*}
\inf_{x\in H_M } \P(\|X^{x}_r\|\le L)>1/2.
\end{equation*}
This combined with  \cite[Lemma~3.8]{HMS11} implies
\begin{equation*}
\inf_{x\in H_M } \P(\|X^{x}_{{\blue t_0}}\|\le \eps/2)\ge\frac12\, \inf_{x\colon \|x\|\le L} \P(\|X^{x}_{{\blue t_0}-r}\|\le \eps/2)>0.
\end{equation*}
Therefore condition \textbf{\ref{A:22}}.1 holds and the statement of the lemma follows from Theorem~\ref{L:contr}(ii).
\end{proof}

Now we can present the proof of Theorem~\ref{T:31}.

\begin{proof}[Proof of Theorem~\ref{T:31}]
 Let us check that all the conditions of Propositions~\ref{Prop:21} are satisfied. Recall the definition of $N_0$ from Lemma~\ref{T:SFDE} {\blue and set $N:=N_0\vee1$}. We take  $E:=\C$, $\rho(x,y)=\|x-y\|\wedge1$, $d(x,y)=({\blue N}\|x-y\|)\wedge1$, $x,y\in E$.

 The first condition of Propositions~\ref{Prop:21} is satisfied by the assumptions of the theorem. The second condition obviously holds. By Lemma~\ref{T:SFDE}(i) there exists $t_0>0$ such that the distance $d$ is contracting for $P_t$ for all $t\ge t_0$. If condition (i) (respectively condition (ii)) of the theorem holds, then by Lemma~\ref{T:SFDE}(ii) (respectively  Lemma~\ref{T:SFDE}(iii)) for any $M>0$ the level set $\{V\le M\}$ is $d$-small for $P_{3r}$.  This implies that the third and the fourth condition of Propositions~\ref{Prop:21} are satisfied.

 Thus, all the conditions of Propositions~\ref{Prop:21} are satisfied. The proof of the theorem is completed by an application of this proposition.
\end{proof}

\section{Exponential ergodicity for SPDEs}\label{S:SPDE}

In this section we develop a general framework   for establishing exponential ergodicity in the SPDE  setting.
For the convenience of the reader, first  we outline the argument and indicate the main difficulty.   Consider, in the spirit of \cite{DaPratoZab}, an SDE in a Hilbert space $H$ of the form
\begin{equation}\label{DZ_eq}
dX(t)=AX(t)\, dt+B(X(t))\, dt+\Sigma(X(t))\, dW(t),\quad t\ge0,
\end{equation}
where $W$ is a cylindrical Wiener process taking values in a Hilbert space $G$; $A$ is a non-positive self--adjoint linear operator $H\to H$ with compact inverse;  $B$, $\Sigma$ are measurable mappings $H\to H$ and $H\to L_2(G,H)$, respectively. Here $L_2(G,H)$ denotes the space of all Hilbert--Schmidt operators $G\to H$.  It follows that $A$ has  negative eigenvalues $-\infty<\hdots\le-\lambda_2\le-\lambda_1<0$ and that the corresponding eigenvectors form a complete orthonormal basis of $H$. We refer to \cite[Chapters~4.2~and~7]{DaPratoZab} for the precise definitions. Assume that $A$, $B$ and $\Sigma$ are such that equation \eqref{DZ_eq} has a unique strong solution.

We see that the principal linear part  of the drift coefficient satisfies
$$
(Ax, x)_H\leq -\lambda_1\|x\|^2_H,\quad x\in H.
$$
The simplest case, e.g. \cite[Chapter~11.6]{DaPratoZab}, is the one where  the non-linear part $B$ of the drift coefficient as well as the diffusion coefficient $\Sigma$ are Lipschitz continuous and their Lipschitz   constants are sufficiently small compared with $\lambda_1$. For such a \emph{dissipative} system one easily gets the following $L_2$-contraction property:  for two solutions $X$, $Y$ of \eqref{DZ_eq} with the same noise and initial conditions $X(0)=x$, $Y(0)=y$,
\begin{equation}\label{contr}
\E\|X(t)-Y(t)\|^2_H\leq e^{-ct}\|x-y\|_H^2,\quad x,y\in H,
\end{equation}
where $c>0$ is some constant that depends only on $\lambda_1$ and  Lipschitz   constants of $B$ and $\Sigma$.
Clearly, \eqref{contr} yields exponential ergodicity of the model.

If the Lipschitz constants of $B$ and $\Sigma$ are large, then the entire system is not dissipative and the $L_2$-contraction property \eqref{contr} for the \emph{true} coupling $(X,Y)$ has no chance to be satisfied. Nevertheless, one can provide an analogue of \eqref{contr} for a  certain  \emph{generalized} coupling using a stochastic control-type argument,  similar to the one which we have used for SFDEs before. Namely, assume for a while that $\Sigma$ is uniformly non-degenerate; that is, the norm of the linear operator $\Sigma(x)^{-1}$ is uniformly bounded over all $x\in H$. Let $X$ be the same as above, and  $Y$ be defined by
\begin{equation}\label{DZ_eq_controlled}
dY(t)=AY(t)\, dt+B(Y(t))\, dt+\Sigma(Y(t))\, dW(t)-\lambda(Y(t)-X(t))\, dt, \quad Y(0)=y.
\end{equation}
If $\lambda>0$ is large enough compared with the Lipschitz constants for $B$, $\Sigma$, then the pair $(X,Y)$ satisfies \eqref{contr}. On the other hand, the SDE for $Y$ can be written as
$$
dY(t)=AY(t)\, dt+B(Y(t))\, dt+\Sigma(Y(t))\, d\wt W(t),
$$
where
$$
d\wt W(t)=dW(t)-\lambda\Sigma(Y(t))^{-1}(Y(t)-X(t))\, dt.
$$
Then the law of $\wt W$ is absolutely continuous with respect to the law of $W$, and  moreover it is possible to bound the total variation distance between these laws and thus to verify Assumptions \textbf{\ref{A:1}} and \textbf{\ref{A:2}} for this generalized coupling. The argument here is essentially the same as in Section \ref{S:SFDE} above. It is based on Theorem \ref{tKLGirsanov} and Lemma \ref{l_TV_via_KL} with a minor difference that now we have to use the analogues of these results for $H$-valued processes. This difference is inessential; see Remark~\ref{rA3}.

If $\Sigma$ fails to be non-degenerate, this argument still applies, but with a certain modification. Namely, let $H_N$ be the span of the first $N$ eigenvectors of $A$ (which correspond to eigenvalues $-\lambda_1, -\lambda_2, \hdots, -\lambda_N$)  and $P_N$ be the projector on this span. Assume that {\blue for all $x\in H$ we have } $\Range\,\Sigma(x) \supset H_N$, and the corresponding pseudo-inverse operator $\Sigma(x)^{-1}:H_N\to G$  is  uniformly bounded over $x\in H$. Instead of \eqref{DZ_eq_controlled} consider the following equation:
\begin{equation}\label{DZ_eq_controlled_N}
dY(t)=AY(t)\, dt+B(Y(t))\, dt+\Sigma(Y(t))\, dW(t)-\frac{\lambda_N}2P_{N}(Y(t)-X(t))\, dt,\,\,\,Y(0)=y.
\end{equation}
Then the previous argument remains applicable under the assumption that the Lipschitz constants of $B$ and  $\Sigma$ are small compared  with $\lambda_N$.
This is essentially the argument developed in \cite{H02}, which combines the dissipativity property of $A$ for high modes of the model with the ``stabilization by noise'' effect for lower modes, {\blue and works well in SPDE models with Lipschitz non-linearities, such as, for example, the
stochastic reaction--diffusion equation, see  \cite[Section~6.1]{H02}.

This approach is still applicable in SPDE models which contain strongly singular  terms, such as the non-linear gradient term $(\u\cdot  \nabla)\u$ in the  Navier--Stokes equation. The technique here dates back (in the deterministic setting) to the exceptional paper \cite{FoPr} and now is used in the theory of finite dimensional attractors  \cite{CF,Te} and meteorology \cite{HoAn,Da,BLSZ}. It was shown in \cite{Mat99} that for the stochastic Navier--Stokes equation a more sophisticated version of \eqref{contr} is available (see \cite[formula (18)]{Mat99}  and \eqref{dissipativity_NS} below), which leads to exponential ergodicity under a certain balance between  energy dissipation and energy influx. In \cite{GMR17} this principal calculation was combined with a stochastic control argument; this led to the  construction of an asymptotic coupling under a milder balance condition which involves only the higher  modes of the system.
This construction appears to be not model specific; in \cite{GMR17}, using this construction, five SPDE models were shown to be uniquely ergodic.}

Here we continue (and in a sense finalize) this argument, and show that essentially the same construction can be used to verify Assumptions \textbf{\ref{A:1}} and \textbf{\ref{A:2}} (and thus to prove exponential ergodicity) for SPDEs with non-Lipschitz non-linearities.

Note that, comparing the proof of Theorem \ref{T:31} and the results of the current section, we can clearly see  the difference between the Assumptions \textbf{\ref{A:22}} and \textbf{\ref{A:2}}. Assumption~\textbf{\ref{A:22}} for SFDEs was verified using the  support theorem.
 For SPDEs the support theorem is not easily available, which makes it difficult to use the argument based on  Assumption~\textbf{\ref{A:22}}. Fortunately, we can use  instead Assumption~\textbf{\ref{A:2}}, which can be verified using the generalized coupling method.

The structure of the rest of this section is as follows. To make the argument visible, we first perform basic  calculation for the two-dimensional stochastic Navier-Stokes equation, which provides a dissipativity-type bound for this model.  Then we present a general toolkit which makes it possible to combine this  bound with an energy-type estimate and verify Assumptions \textbf{\ref{A:1}} and \textbf{\ref{A:2}}. Finally,  we apply this general toolkit to three other SPDE models from \cite{GMR17}: the hydrostatic Navier-Stokes model, the fractionally dissipative Euler model, and the damped Euler-Voigt model. Additionally, we also treat the Boussinesq equations.

Note that the remaining SPDE from \cite{GMR17}, the damped nonlinear wave equation, does not have any non-Lipschitz nonlinearities; therefore the exponential ergodicity of this SPDE can be shown directly by an argument similar to the one used in \cite{H02}. Thus we do not treat it here.

\subsection{2D  Navier-Stokes Equation, I: basic calculations}\label{SS:TDNS}

In this and the next sections we will frequently use the following notation: for a function $f\colon\R_+\to\R$ we denote the part of the trajectory
\begin{equation}\label{conventiontraj}
f_{[0,t]}:=\{f(s),\,s\in[0,t]\},\quad t\ge0.
\end{equation}

Recall the standard notation. Denote by $V$ the subspace of $H^1(\D)^2$,  which contains $\mathbf{u}$ such that $\nabla \cdot \mathbf{u}  = 0$ and $\u\evalat{ \d D} = 0$.
Denote by $H$ the completion of $V$ w.r.t. the $L_2(\D)^2$-norm, by $P_H$ the projector in $L_2(\D)^2$ on $H$, and by $A:=-P_H\Delta$
the \emph{Stokes operator}. The eigenvectors $e_1, e_2, \dots$ of the Stokes operator (with the corresponding eigenvalues $0<\lambda_1\le \lambda_2\le\dots$) form a complete orthonormal basis of $H$. We denote by $\|\cdot\|_H$ the standard $L_2(\D)^2$-norm and for $\mathbf w=(w_1,w_2)\in V$ put  $\|\mathbf{w}\|_V^2:=\|\nabla w_1\|_H^2+\|\nabla w_2\|_H^2$.

Consider the 2D stochastic Navier--Stokes equation evolving on a bounded domain $\D\subset\R^2$ with a smooth boundary  $\d \D$:
\begin{align}\label{SNS}
&\begin{aligned}[b]
d\u(z,t)+(\u(z,t)\cdot  \nabla)\u(z,t)dt&=(\nu \Delta\u(z,t) -\nabla p(z,t)
+\mathbf{f}(z))\,dt\\
&\phantom{=}+\sum_{k=1}^m \bsigma_k(z)dW^k(t),\quad z\in \D,\,t\ge0;
\end{aligned}\\
&\u(\cdot,0)=\bx,\,\,\nabla\cdot \u=0,\,\,\u\evalat{ \d D} = 0,\nn
\end{align}
where $\mathbf{u} = (u_1, u_2)$ is the unknown velocity field, $p$ is the unknown pressure, $m\in\mathbb{N}$, $W=(W^1, W^2, \hdots, W^m)$ is a standard $m$--dimensional Brownian motion, $\mathbf{f}, \bsigma_1, \dots, \bsigma_m\in L_2(D)^2$, $\nu>0$. {\blue As usual  (see, e.g., \cite[Remark 3.1]{GMR17}), without loss of generality, we can assume that $\mathbf{f}$ and all $\bsigma_i$ are divergence free and, thus, are in $H$.}

It is known (\cite[Section 3.1.1]{GMR17}) that for any initial condition $\bx\in H$ this equation has a unique strong solution, which in the case of ambiguity will be denoted later by $\u^\bx$. Further, $\u$ is a Feller Markov process with  state space $H$.

The generalized coupling construction which we will use is the same as in \cite[Section~6.2.1]{KS17}, and is a slight variation of the construction from \cite[Section 3.1.2]{GMR17}. Namely, we consider $\sigma$ as a linear operator $\R^m\to L_2(\D)^2$ and fix the maximal possible $N$ such that
\begin{equation}\label{range}
H_N:=P_NH\subset\Range\, (\sigma)=\Span (\bsigma_k,\, k=1, \dots, m);
\end{equation}
here $P_N$ stands for the projection to the span of the first $N$ eigenvectors $e_1, e_2,\hdots e_N$ of the Stokes operator.

For given $\bx,\by\in H$, take $\bX^{\bx,\by}:= \u^\bx$ and define  $\bY^{\bx,\by}$ as the solution to the following  stochastic 2D Navier--Stokes equation \eqref{SNS} with the initial condition $\by$ and the additional control term:
\begin{align}\label{2dNS1a}
&\begin{aligned}[b]
d\bY^{\bx,\by}(z,t)+(\bY^{\bx,\by}(z,t)\cdot  \nabla)\bY^{\bx,\by}(z,t)dt&=(\nu \Delta\bY^{\bx,\by}(z,t) -\nabla\wt p(z,t)+\mathbf{f}(z))\,dt\\
&\phantom{=}+\frac{\nu\lambda_{N+1}}2P_N(\bX^{\bx,\by}(\cdot,t)-\bY^{\bx,\by}(\cdot,t))dt\\
&\phantom{=}+\sum_{k=1}^m \bsigma_k(z)dW^k(t),
\end{aligned}\\
&\bY^{\bx,\by}(z,0)=\by(z),\,\,\nabla\cdot \bY^{\bx,\by}=0,\,\,{\bY^{\bx,\by}}\evalat{ \d D} = 0.\nn
\end{align}
Note that \eqref{2dNS1a} is just \eqref{DZ_eq_controlled_N} up to a proper change of notation.
There are two different ways to understand this equation  as a modification of \eqref{SNS}, both of them being useful for particular purposes. First, one can interpret \eqref{2dNS1a} as an analogue of  \eqref{SNS} with the operator $\Delta$ changed to $\widehat\Delta:=\Delta-(\nu\lambda_{N+1}/2)P_N$, and with an additional forcing term $(\nu\lambda_{N+1}/2) P_N\mathbf{u^\bx}$. This allows one to apply the Girsanov theorem to show that the strong solution to  \eqref{2dNS1a} is uniquely defined; see \cite[Remark 8]{KS17} for a detailed exposition. Second, we can write \eqref{2dNS1a} in the form   \eqref{SNS} with $\bx$ changed to $\by$, and with $dW(t)$ replaced by
\begin{equation}\label{newWdr}
\dd W^{\bx,\by}(t):=\dd W(t)+\beta^{\bx,\by}(t)\, \dd t, \quad \beta^{\bx,\by}(t):=\frac{\nu\lambda_{N+1}}2\sigma^{-1}P_N (\bX^{\bx,\by}(t)-\bY^{\bx,\by}(t)).
\end{equation}

By \eqref{range}, the pseudo--inverse operator $\sigma^{-1}\colon H_N\to \R^m$ is well defined and bounded; thus  there exists a constant $C>0$ such that for all $t\ge0$
\begin{equation}\label{beta_XY}
\|\beta^{\bx,\by}(t)\|_{\mathbb{R}^m}{\blue \leq C \|P_N(\bX^{\bx,\by}(t)-\bY^{\bx,\by}(t))\|_{H}}\leq C \|\bX^{\bx,\by}(t)-\bY^{\bx,\by}(t)\|_{H}.
\end{equation}

To check the first condition of Assumption~\textbf{\ref{A:1}}, we note that by construction we have $\Law(\bX^{\bx,\by})=\P_x$.
Further, recall that for any $t\ge0$ the strong solution to equation \eqref{SNS} with the initial value $\by$, $\u^\by(t)$, is an image of the driving noise under a measurable mapping
$$
\Phi_t^\by: \C([0, t], \mathbb{R}^m)\to H.
$$
In other words, $\u^\by(t)=\Phi^\by_t(W_{[0,t]})$, recall the convention \eqref{conventiontraj}.
It follows from the Girsanov theorem (\cite[Theorem~7.4]{LS}) that $\Law (W_{[0,t]}^{x,y})$ is absolutely continuous with respect to $\Law (W_{[0,t]})$. Therefore, by the uniqueness of the solution, we have $\bY^{\bx,\by}(t)=\Phi^\by_t(W_{[0,t]}^{x,y})$.

{\blue By the mere definition of the total variation distance,  
\begin{align*}
 d_{TV}\bigl(P_t(\by,\cdot), \Law (\bY^{x,y}(t))\bigr)&=d_{TV}\bigl(\Law(\u^\by(t)), \Law (\bY^{x,y}(t))\bigr)\nn\\
 &=d_{TV}\bigl(\Law(\Phi^\by_t(W_{[0,t]})), \Law (\Phi^\by_t(W_{[0,t]}^{x,y}))\bigr)\nn\\
 &\le d_{TV}\bigl(\Law (W_{[0,t]}),\Law (W^{x,y}_{[0,t]})\bigr).\\
\end{align*} 
Thus,  for any $\delta\in (0,1]$ we derive that there exists $C_\delta>0$ such that for any $t\ge0$ we have by  \eqref{beta_XY} and Theorem~\ref{L:e2e4}
\be\label{TV_bound_solutions}
 d_{TV}\bigl(P_t(\by,\cdot), \Law (\bY^{x,y}(t))\bigr) \le C_\delta \Bigl(\E \bigl(\int_0^t \|\bX^{\bx,\by}(s)-\bY^{\bx,\by}(s)\|_{H}^2\, \dd s\bigr)^\delta\Bigr)^{1/(1+\delta)}.
\ee
}

As we have already explained, the crucial difficulty in the estimation of the  $\|\cdot\|_H$-difference between $\bX^{\bx,\by}$ and $\bY^{\bx,\by}$ appears because of the  non-Lipschitz structure of the term $(\u\cdot  \nabla)\u$ in the original equation. To overcome this difficulty, we use the idea {\blue from \cite{FoPr}, see also \cite[Section 3.1.2]{GMR17}.  The Ladyzhenskaia trick \cite[Formula (6)]{L59} yields the following generic bound:}
\begin{equation}\label{SobolevNS}
\left|\int_{\D}( \mathbf{v}\cdot \nabla) \mathbf{u}\cdot \mathbf{v}\, \dd z\right|
\leq {\blue 2} \|\mathbf{v}\|_H\|\mathbf{v}\|_V\|\mathbf{u}\|_V, \quad \mathbf{u}, \mathbf{v}\in V.
\end{equation}

Combining the It\^o formula, the Poincar\'e inequality, bound  \eqref{SobolevNS}, and the Gronwall inequality, one gets
for any $t\ge0$
\begin{equation}\label{dissipativity_NS}
\|\bX^{\bx,\by}(t)-\bY^{\bx,\by}(t)\|^2_H\le  \|\bx-\by\|^2_H\exp\Bigl(-\nu\lambda_{N+1} t+\frac{ {\blue 4}}{\nu}\int_0^t\|\bX^{\bx,\by}(s)\|_V^2\, d s\Bigr),
\end{equation}
 see \cite[formula (6.10)]{KS17} and \cite[formula (3.9)]{GMR17}. This inequality can be understood as a proper substitute for the  dissipativity bound \eqref{contr}. However, due to the extra term, which involves the stronger $\|\cdot\|_V$-norm of the solution $\bX^{\bx,\by}$, this bound can not be directly used  to verify Assumption \textbf{\ref{A:1}}.2. Fortunately, the extra term can be efficiently bounded using the energy estimates. We have (see,  \cite[p. 627, line~16]{GMR17})
\begin{equation*}
d \| \u \|^2_H + 2\nu \| \u\|_V^2 dt = 2(\mathbf{f}, \u)_H dt +  \|\sigma\|^2_{H} dt +  2 (\sigma, \u)_H dW.
\end{equation*}
Using the Cauchy inequality,
$$
2(\mathbf{f}, \u)_H\leq \nu\|\u\|_V^2+\frac1\nu\|A^{-1/2}\mathbf{f}\|^2_H,
$$
we obtain the integral estimate
\begin{equation}\label{energy_NS}
\| \u(t) \|^2_H+\nu\int_0^t\|\u(s)\|^2_V\, ds\leq \| \u(0) \|^2_H+ \bigl(\frac1\nu\|A^{-1/2}\mathbf{f}\|^2_H+\|\sigma\|^2_H\bigr)t+M(t),
\end{equation}
where $M$ is a continuous local martingale with
\begin{equation}\label{QV_NS}
d\langle M\rangle_t=4 (\sigma, \u)_H^2\, dt\le 4 \|\sigma\|^2_H\|\u\|^2_Hdt\le C  \|\sigma\|^2_H\|\u\|^2_Vdt,
\end{equation}
where $C>0$ and the last inequality follows from the Poincar\'e inequality.

With the estimates \eqref{TV_bound_solutions} and \eqref{dissipativity_NS} -- \eqref{QV_NS} in hand, we are able to construct a premetric $\theta$ {\blue(which is going to be non-symmetric due to the presence of the extra factor in \eqref{dissipativity_NS})}  such that Assumptions \textbf{\ref{A:1}} and \textbf{\ref{A:2}} are verified. This allows to prove exponential ergodicity of $\u^\bx$. Such a construction is quite generic and can be used in various SPDE models. Thus, for the convenience of further applications,  we introduce it separately and in an abstract form.

\subsection{A general toolkit for SPDEs}\label{s42}

In this subsection, motivated by the above analysis, we introduce a general framework for establishing exponential ergodicity of the solutions of SPDEs. {\blue We use the general setting from Section~\ref{S:11}; that is, $(E, \rho)$ is a Polish space, $\{P_t (x, A), x\in E, A\in \mathcal{E}\}_{t\in \R_+}$ is a Markov transition function and $\{\P_x, x\in E\}$ is the corresponding Markov family. Recall the notion of the premetric that was defined in the beginning of Section ~\ref{S:MR}.} Assume the following.

\begin{Assumption}{H}
There exist a lower semicontinuous function $U\colon E\to [0, \infty)$, a measurable function $S\colon E\to [0, \infty]$ and a premetric $q$ on $E$  such that for any given $x,y\in E$ there exists a couple  of progressively measurable random processes $X^{x,y}=(X^{x,y}_t)_{t\ge0}$, $Y^{x,y}=(Y^{x,y}_t)_{t\ge0}$, that satisfies the following set of conditions:
\begin{itemize}
\item[\textbf{H1}] (dissipativity bound):
 \begin{equation}\label{diss_abstract}
  q(X_t^{x,y}, Y^{x,y}_t)\leq q(x,y)\exp\left(-\zeta t+\kappa \int_0^t S(X^{x,y}_s)\, ds\right),\quad t\ge0
  \end{equation}
  where $\zeta>0$, $\kappa\ge0$;

  \item[\textbf{H2}] (energy estimate):
\begin{equation}\label{energy_abstract}
U(X^{x,y}_t)+\mu\int_0^tS(X^{x,y}_s)\, ds\leq U(X^{x,y}_0)+bt+M_t,\quad t\ge0,
\end{equation}
where $\mu>0$, $b\ge0$ are some constants such that
\begin{equation}\label{condtheta}
\zeta>{\frac{\kappa b}\mu};
\end{equation}

and $M$ is a continuous local martingale with $M_0=0$ and
\begin{equation}\label{martingale_abstract}
d\langle M\rangle_t\leq b_1 S(X^{x,y}_t)\, dt+b_2 dt,\quad t\ge0,
\end{equation}
where $b_1\ge0$, $b_2\ge0$.

\item[\textbf{H3}] (error-in-law bounds):  $\Law\, (X^{x,y})=\P_x$ and for every $\delta\in (0,1]$ there exists a constant $C_\delta>0$ such that for any $t>0$
\begin{equation}
\label{dtvbound}
d_{TV}\bigl(\Law (Y_t^{x,y}), P_t(y,\cdot)\bigr)\leq C_\delta \Bigl(\E \bigl(\int_0^t  q(X_s^{x,y}, Y^{x,y}_s)\, \dd s\bigr)^\delta\Bigr)^{1/2}.
\end{equation}
Additionally, for any $\delta\in (0,1]$, $R>0$ there exists $\eps=\eps(R,\delta)>0$ such that\\
$\E\bigl(\int_0^t  q(X_s^{x,y}, Y^{x,y}_s)\, \dd s\bigr)^\delta<R$ implies
\begin{equation}
\label{dtvbound2}
d_{TV}\bigl(\Law (Y_t^{x,y}), P_t(y,\cdot)\bigr)\leq 1-\eps.
\end{equation}

\end{itemize}
The functions $U$, $S$, $q$ as well as the constants $\zeta, \kappa, \mu, b, b_1, b_2, C_\delta, \eps$ should not depend on the pair $x,y$.
 \end{Assumption}

 \medskip

 Typically, the process $Y^{x,y}$ is constructed as the solution to the original SPDE where the Brownian motion is replaced by the Brownian motion with a suitable drift that pushes $Y^{x,y}$ towards $X^{x,y}$  (recall \eqref{newWdr} and the corresponding construction for the stochastic Navier--Stokes equation).  In this case condition \textbf{H3} can be replaced by the following simple condition on this drift.  Recall convention \eqref{conventiontraj}.

 \begin{Lemma}\label{L:pomoshH3}
 Let $W$ be an $m$-dimensional Brownian motion, $m\ge1$. Assume that there exists a constant $c>0$ such that for each $t\ge0$, $x,y\in E$ there exists a measurable function $\Phi=\Phi^{t,x,y}\colon \C[0,t]\to E$ and {\blue progressively} measurable processes $\beta^{x,y},\xi^{x,y}\colon\Omega\times[0,t]\to\R^m$ such that
 \begin{itemize}
  \item[\rm{1.}] We have $d\xi^{x,y}_s=dW_s+\beta^{x,y}_s ds$, $s\in[0,t]$.
  \item[\rm{2.}]  $ \Law(\Phi(W_{[0,t]}))=P_t (y,\cdot)$ and $\Phi(\xi_{[0,t]})=Y^{x,y}_t$.
  \item[\rm{3.}] For each $s\in[0,t]$ we have $|\beta_s|^2\le c q(X_s^{x,y}, Y^{x,y}_s)$.
 \end{itemize}

 Then  conditions \eqref{dtvbound} and \eqref{dtvbound2}  hold.
 \end{Lemma}
\begin{proof}
Fix $\delta\in(0,1]$, $t>0$. Denote
$$
M_\delta:=\E \bigl(\int_0^t |\beta_s|^2 ds\bigr)^\delta.
$$
It follows from condition 3 of the lemma that
\begin{equation}\label{Mdeltaest}
M_\delta\le c^\delta \E \bigl(\int_0^t q(X_s^{x,y}, Y^{x,y}_s) ds\bigr)^\delta.
\end{equation}

Further, inequality \eqref{TV_bound_one_delta} from Theorem~\eqref{L:e2e4} and condition 2 of the lemma imply
\begin{align*}
d_{TV}\bigl(\Law (Y_t^{x,y}), P_t(y,\cdot)\bigr)&=
d_{TV}\bigl(\Law (\Phi(\xi_{[0,t]})), \Law(\Phi(W_{[0,t]}))\bigr)\\
&\le d_{TV}\bigl(\Law (\xi_{[0,t]}), \Law(W_{[0,t]})\bigr)\\
&\le  2^{(1-\delta)/(1+\delta)}(M_{\delta})^{1/(1+\delta)}.
\end{align*}
Combining this with \eqref{Mdeltaest}, we obtain \eqref{dtvbound}.

Similarly,  inequality \eqref{TV_bound_two_delta} from Theorem~\eqref{L:e2e4}, condition 2 of the lemma  and \eqref {Mdeltaest} yield  \eqref{dtvbound2}.
\end{proof}

Now let us present the main result of this section. It shows that Assumptions  \textbf{H1}--\textbf{H3} together with the existence of a suitable   Lyapunov function imply exponential ergodicity.

\begin{Theorem}\label{T:lyapSPDE}
Suppose that the Markov kernel $(P_t)$ is Feller and satisfies Assumptions  \textbf{H1}--\textbf{H3} for some functions $U$, $S$, $q$. Assume further that there exists a measurable function $V\colon E\to\R_+$ such that for some $\gamma>0$, $K>0$ we have
\begin{equation}\label{Llyap}
\E_x V(X_t)\le V(x)-\gamma\E_x\int_0^t V(X_s)\, ds +Kt,\quad t\ge0,\,x\in E
\end{equation}
and  for any $M>0$ the functions $U(\cdot)$, $q(\cdot,\cdot)$ are bounded on the level sets $\{V\le M\}$ and  $\{V\le M\}\times\{V\le M\}$, respectively.

If additionally $\rho\le  q^\delta$ for some $\delta>0$, then $(P_t)$ has a unique invariant measure $\pi$. Moreover,  there exist constants $C,r>0$ such that
\begin{equation*}
W_{\rho\wedge1}(P_t(x,\cdot),\pi) \le C (1+V(x))e^{-rt },\quad t\ge0,\,x\in E.
\end{equation*}
\end{Theorem}

The proof of Theorem~\ref{T:lyapSPDE} is based on the following key lemma. It shows that Assumptions  \textbf{\ref{A:1}} and \textbf{\ref{A:2}} follow from Assumptions \textbf{H1}--\textbf{H3}.

\begin{Lemma}\label{P:41} Suppose that Assumptions \textbf{H1}--\textbf{H3} hold. Then there  exists $\alpha_0>0$ that
depends only on $\zeta, \kappa, \mu, b, b_1, b_2, C_\delta, \eps$,  such that for any $\alpha\in (0, \alpha_0]$ there exist
constants $C, Q, \lambda>0$ such that
\begin{itemize}
 \item[\rm{(i)}] Assumption \textbf{\ref{A:1}} holds for the premetric
$$
\theta(x,y)=e^{Q U(x)}q(x,y)^{\alpha},
$$
and the rate functions $L(t):=C$, $r(t)=C\exp(-\lambda t)$.
\item[\rm{(ii)}] For any set $B$ such that $U(\cdot)$ is bounded on $B$, and $q(\cdot, \cdot)$ is bounded on $B\times B$ there exists  $C>0$ such that Assumption \textbf{\ref{A:2}} holds for the same premetric $\theta$, the set $B$ and the rate function $R(t):=C\exp(-\lambda t)$.
\end{itemize}
\end{Lemma}

\begin{proof}[Proof of Lemma~\ref{P:41}]
(i). Take any $x,y\in E$. First of all let us derive a good bound on $q(X_t^{x,y}, Y^{x,y}_t)$, $t\ge0$.
Let $\gamma>0$ be a sufficiently small parameter to be chosen later. Define
$$
\Xi_\gamma:=\sup_{t\ge0}(M_t-\gamma\langle M\rangle_t).
$$
By the Dambis--Dubins--Schwarz theorem (see, e.g., \cite[Theorem~5.1.6]{RY}), there exists (maybe on an extended probability space) a Brownian motion $\wh W$ such that $M_t=\wh W_{\langle M\rangle_t}$, $t\ge0$. Therefore,
$$
\Xi_\gamma\leq \sup_{t\ge0}(\wh W_t-\gamma t)
$$
and
\begin{equation}\label{formulaxigamma}
\P(\Xi_\gamma\ge R)\le\P\bigl(\sup_{t\ge0}(\wh W_t-\gamma t)\ge R\bigr)=e^{-2\gamma R},\quad R\ge0,
\end{equation}
see \cite[Part II, formula 2.0.2.(1)]{Borodin_Salminen}. By \eqref{energy_abstract} and \eqref{martingale_abstract} we have
\begin{equation}\label{uravnN}
U(X^{x,y}_t)+(\mu-\gamma b_1)\int_0^tS(X^{x,y}_s)\, ds\leq U(x)+(b+\gamma b_2)t+\Xi_\gamma.
\end{equation}
From now on we assume that $0<\gamma<\mu/b_1$. Then, combining \eqref{uravnN} with \eqref{diss_abstract}, we obtain
\begin{equation}\label{qqq}
q(X_t^{x,y}, Y^{x,y}_t)\le q(x,y)\exp\Big(-\chi t+\upsilon\bigl(U(x)-
U(X^{x,y}_t)+\Xi_\gamma\bigr)\Big),
\end{equation}
where we denoted
\begin{equation*}
\upsilon=\frac{\kappa}{\mu-\gamma b_1},\quad\chi:=\zeta-\frac{\kappa (b+\gamma b_2)}{\mu-\gamma b_1}.
\end{equation*}
Recall that thanks to \eqref{condtheta} we have $\zeta- \kappa b/\mu>0$. Therefore we can take $\gamma$  small enough so that $\chi>0$.

We begin with verifying \textbf{\ref{A:1}}.1. By \textbf{H3}, we have $\Law(X^{x,y})=\P_x$. To estimate the distance between $\Law(Y^{x,y}_t)$ and $P_t(y,\cdot)$ we also use \textbf{H3} together with 
\eqref{qqq}. We get for any $\delta\in(0,1)$
\begin{align}\label{qqqq}
d_{TV}(\Law(Y^{x,y}_t),P_t(y,\cdot))
&\le C_\delta \Bigl(\E \bigl(\int_0^t  q(X_s^{x,y}, Y^{x,y}_s)\, \dd s\bigr)^\delta\Bigr)^{1/2}\nn\\
&\le C q(x,y)^{\delta/2}e^{\upsilon\delta U(x)/2}\E\exp\bigl(\frac{\upsilon\delta}{2}\Xi_\gamma\bigr).
\end{align}
Define now for $\alpha\in(0,1)$
\begin{equation}\label{rhodef}
\theta_\alpha(z_1,z_2):= q(z_1,z_2)^{\alpha} e^{\alpha \upsilon U(z_1)},\quad z_1,z_2\in E.
\end{equation}

Recall that \eqref{formulaxigamma} implies that ${\E \exp(K \Xi_\gamma)<\infty}$ for $K<2 \gamma$. Thus, by taking
\begin{equation}\label{alphazero}
\alpha_0:=(\gamma/\upsilon)\wedge(1/2),
\end{equation}
we see that for any $\alpha\in(0,\alpha_0]$ inequality \eqref{qqqq} implies
\begin{equation*}
d_{TV}(\Law(Y^{x,y}_t),P_t(y,\cdot))\le C \theta_{\alpha}(x,y)
\end{equation*}
and therefore Assumption \textbf{\ref{A:1}}.1 holds for the premetric $\theta_\alpha$, $\alpha\in(0,\alpha_0]$.

To check \textbf{\ref{A:1}}.2 we employ again \eqref{qqq}.
Using the definition of the premetric $\theta_\alpha$ in \eqref{rhodef},
we derive
\begin{align}\label{A12perv}
\E \theta_\alpha(X_t^{x,y},Y_t^{x,y})&=  \E q(X_t^{x,y},Y_t^{x,y})^{\alpha}
e^{\alpha \upsilon U(X_t^{x,y})}\nn\\
&\le \E e^{-\lambda t}q(x,y)^{\alpha}e^{\alpha\upsilon U(x)}e^{\alpha \upsilon \Xi_\gamma}\nn\\
&\le C  e^{-\lambda t}\theta_\alpha(x,y),
\end{align}
where we denoted $\lambda=\alpha \chi$. Thus Assumption \textbf{\ref{A:1}}.2 also holds for the premetric $\theta_\alpha$, $\alpha\in[0,\alpha_0]$.

(ii). We will use the same notation as in the part (i) of the proof. Fix $\alpha\in(0,\alpha_0]$ and take any $x,y\in B$. First let us verify \textbf{\ref{A:2}}.1. Clearly, $\Law(X^{x,y})=\P_x$ thanks to \textbf{H3}. Fix now $t>0$. To estimate the total variation distance between $\Law(Y_t^{x,y})$ and $P_t(y,\cdot)$ denote
\begin{equation*}
M_\alpha:=\E\left(\int_0^t  q(X_s^{x,y}, Y^{x,y}_s)\, \dd s\right)^\alpha.
\end{equation*}
Inequality \eqref{qqq} and the boundedness of $q$ and $U$ on the set $B$ imply
\begin{equation*}
M_\alpha\le C \E \exp(\upsilon \alpha \Xi_{\gamma})\le C_1,
\end{equation*}
where we also used \eqref{alphazero} and the fact that $\alpha\le \alpha_0$. Now condition \eqref{dtvbound2} from \textbf{H3}  implies
\begin{equation*}
d_{TV}(\Law(Y^{x,y}_t),P_t(y,\cdot))\le 1-\eps
\end{equation*}
for some $\eps>0$. This yields \eqref{TV_B}.

To verify \textbf{\ref{A:2}}.2 we use \eqref{A12perv} and the definition of the metric $\theta_\alpha$, which is \eqref{rhodef}. Using again the fact that $q$ and $U$ are bounded on $B$, we get
\begin{equation*}
\E \theta_\alpha(X_t^{x,y},Y_t^{x,y})\le  C e^{-\lambda t},
\end{equation*}
for some $C=C(B)>0$. Thus, \textbf{\ref{A:2}}.2 holds for the metric $\theta_\alpha$.
\end{proof}

Now we can present a proof of Theorem~\ref{T:lyapSPDE}.
\begin{proof}[Proof of Theorem~\ref{T:lyapSPDE}]
The theorem follows from Lemma~\ref{P:41} and Theorem~\ref{T:Lyap2}.

Recall the definition of $\alpha_0$ from Lemma~\ref{P:41}. Take $\alpha:=\alpha_0\wedge\delta$. Then by  Lemma~\ref{P:41} there exists $Q>0$ such that  Assumptions~\textbf{\ref{A:1}} {\blue and \textbf{\ref{A:2}}} holds for the premetric $\theta(x,y):=e^{QU(x)}q(x,y)^\alpha$.

Let us now check that all the conditions of Theorem~\ref{T:Lyap2} are satisfied for this premetric $\theta$. The first condition of Theorem~\ref{T:Lyap2} is satisfied due to \eqref{Llyap}. Since, by assumption,  $\rho\le  q^\delta$ and ${\blue U}(x)\ge0$ for any $x\in E$, we have
$$
\rho(x,y)\wedge 1\le q(x,y)^\delta{\blue\wedge 1}\le q(x,y)^\alpha{\blue\wedge 1}\le \theta(x,y),\quad x,y\in E,
$$
and thus the second condition of Theorem~\ref{T:Lyap2} holds.

The third condition of  Theorem~\ref{T:Lyap2} holds by Lemma~\ref{P:41}(i).

Finally, since $U$ and $q$ are bounded on the level sets of $V$, we see that  Lemma~\ref{P:41}(ii) implies that the fourth condition  of  Theorem~\ref{T:Lyap2} is also satisfied. Now the statement of the theorem follows immediately from Theorem~\ref{T:Lyap2}.
\end{proof}

\subsection{2D  Navier-Stokes Equation, II: exponential ergodicity}

Now with the general toolkit in hand we can proceed directly and establish exponential ergodicity of the 2D stochastic Navier--Stokes equation. Note that exponential ergodic bounds had been obtained previously for the Navier-Stokes equation only on a 2-dimensional torus, see \cite{Mat2002} and  \cite{HM08}. Unlike these results, Theorem \ref{tNS} does not rely  on the geometry of the domain ${D}$, or on the fine structure of the forcing. This illustrates well,  we believe, that the generalized coupling method is quite insensitive with respect to the local structure of the model.

%

Recall that $\lambda_i$ is the $i$th largest eigenvalue of the Stokes operator $A=-P_H\Delta$.

\begin{Theorem}\label{tNS}
Assume that there exists $N\in\Z_+$ such that
$$
P_NH\subset\Span (\bsigma_k,\, k=1, \dots, m)
$$
and
\begin{equation}\label{granica}
\lambda_{N+1}>{\blue 4\nu^{-4}\|A^{-1/2}\mathbf{f}\|^2_H+4\nu^{-3}\|\sigma\|^2_H.}
\end{equation}
Then the stochastic Navier--Stokes equation \eqref{SNS} has a unique invariant measure $\pi$. Further,  there exist constants $C>0$, $r>0$ such that
\begin{equation*}
W_{\|\cdot\|_H\wedge1}(\Law (\u^\bx(t)),\pi) \le C (1+\|\bx\|_H^2)e^{-rt },\quad t\ge0,\,\bx\in H.
\end{equation*}
\end{Theorem}
\begin{proof}
Let us check that all the conditions of Theorem~\ref{T:lyapSPDE} are satisfied. Motivated by the preliminary analysis in Section~\ref{SS:TDNS}, let us take
$$
U(\bx):=\|\bx\|_H^2, \quad S(\bx):=\|\bx\|_V^2, \quad q(\bx,\by):=\|\bx-\by\|_H^2, \quad \bx,\by\in H.
$$
Then it follows from \eqref{dissipativity_NS} that Assumption \textbf{H1} is satisfied with $\zeta:=\nu\lambda_{N+1}$ and $\kappa:={\blue 4}/\nu$.

Further, we see from bounds \eqref{energy_NS} and \eqref{QV_NS} that  \textbf{H2} also holds with $\mu:=\nu$ and
$b:=\|A^{-1/2}\mathbf{f}\|^2_H/\nu+\|\sigma\|^2_H$. Constraint \eqref{condtheta} follows from \eqref{granica}.

To verify \textbf{H3} we employ Lemma~\ref{L:pomoshH3}. As discussed in Section~\ref{SS:TDNS}, for any $t\ge0$ there exists some measurable mapping $\Phi$ such that $\bY^{\bx,\by}(t)=\Phi_t^\by(W^{\bx,\by}_{[0,t]})$ and $\u^\by=\Phi_t^\by(W_{[0,t]})$. Inequality
\eqref{beta_XY} implies that for some $C>0$
$$
|\beta_s|^2\le C q(\bX^{\bx,\by}(s),\bY^{\bx,\by}(s)),\quad s\in[0,t].
$$
Hence, by Lemma~\ref{L:pomoshH3} condition  \textbf{H3} also holds.

Finally, we introduce the following Lyapunov function: $V(\bx):=\|\bx\|_H^2$. By definition, for any $M>0$ the functions $U$ and $q$ are bounded on the level sets $\{V\le M\}$ and $\{V\le M\}\times\{V\le M\}$, respectively. Inequality
\eqref{Llyap} follows from the energy estimate \eqref{energy_NS} {\blue and the Poincar\'e inequality $\|u\|_V\ge C\|u\|_H$}.

Since $\rho:=\|\bx-\by\|_H\le q(\bx,\by)^{1/2}$, we see that all conditions of Theorem~\ref{T:lyapSPDE} hold. This immediately implies the statement of the theorem.
\end{proof}

\subsection{2D Hydrostatic Navier-Stokes Equation}
In this section we treat the stochastic hydrostatic Navier-Stokes equations. Fix $L,h>0$ and consider a domain $D\subset\R^2$ defined by $D:=\{(z_1,z_2)\colon z_1\in(0,L),\,z_2\in(-h,0)\}$. Introduce the lateral side of the boundary $\Gamma_l:=\{0,L\}\times[-h,0]$ and the horizontal side $\Gamma_h:=[0,L]\times\{-h,0\}$. Consider the following equations for an unknown velocity field $(u,w)$ and pressure $p$ on $D$
\begin{align}
&du+  (u \partial_{z_1} u + w \partial_{z_2} u) dt=(\nu\Delta u -\partial_{z_1} p)dt+  \sum_{k=1}^m \sigma_k dW^k,\label{2dhns}\\
& \partial_{z_2} p = 0,\,\,\partial_{z_1} u + \partial_{z_2} w  = 0,\,\,u(0)=x,\,\,\nn\\
&  u\evalat{ \Gamma_l}\! = 0,\,\,\d_{z_2}u\evalat{ \Gamma_h}\! =w\evalat{ \Gamma_h}\!=0.\label{bc2dhns}
\end{align}
where $W=(W^1,W^2,\dots,W^m)$ is a standard $m$--dimensional Brownian motion,  $\nu>0$. We assume that for any $k=1,\hdots, m$ we have  $\sigma_k\in H^2(D)$, ${\sigma_k}\evalat{ \Gamma_l}\! = 0$, $\d_{z_2} {\sigma_k}\evalat{ \Gamma_h}\!=0$, $\int_{-h}^0 \sigma_k dz_2\equiv 0$.
Denote
\begin{align*}
&H:=\{\phi\in L_2(D)\colon \int_{-h}^0\phi(z_1,z_2)\,dz_2=0\, \text{ for all $z_1\in(0,L)$}\},\\
&V:=\{\phi\in H^1(D)\colon \int_{-h}^0\phi(z_1,z_2)\,dz_2=0\, \text{ for all $z_1\in(0,L)$  and  $\phi\evalat{\Gamma_l}=0$}\}.
\end{align*}
Let $\|\cdot\|_H$ be the standard $L_2(\D)$-norm and for $\phi\in V$ put  $\|\phi\|_V^2:=\|\d_{z_1} \phi\|_H^2+\|\d_{z_2} \phi\|_H^2$. We consider the following complete orthonormal basis of $H$:
$$
\bigl\{e_{i,j}:=\frac{2}{\sqrt{hL}}\sin(\frac{i\pi z_1}L) \cos(\frac{j\pi z_2}h)\bigr\}_{i,j\ge1},
$$
which corresponds to eigenvectors of the negative Laplacian operator with boundary conditions given by \eqref{bc2dhns}. Their associated eigenvalues are
$$
\bigl\{\lambda_{i,j}:=\pi^2(i^2/L^2+j^2/h^2)\bigr\}_{i,j\ge1}.
$$
It is known (\cite[Section~3.2.1]{GMR17}, see also \cite[Theorems~1.3 and 1.5]{GKVZ14}) that under the above assumptions on $\sigma$, for any $x\in V$ stochastic hydrostatic Navier-Stokes equation \eqref{2dhns} has a unique strong solution, which in the case of ambiguity will be further denoted by $u^x$. Moreover, $u$ is a Feller Markov process with the state space $V$ and for Lebesgue-almost all $t>0$ the random element {\blue$u^x(t)\in \{\phi\in H^2(D):\d_{z_2}\phi\evalat{ \Gamma_h}=0\}$}. The existence and uniqueness of the invariant measure of $u$ was shown in \cite[Section~3.2]{GMR17}. The weak convergence of transition probabilities to the invariant measure was established in \cite[Section~6.2.2]{KS17}. We strengthen these results and prove exponential ergodicity of $u$.

\begin{Theorem} \label{T:2dhns}
Assume that
\begin{equation}\label{nondeg2dhns}
\Span (e_{i,j}:\,\lambda_{i,j}\le K )\subset\Span (\sigma_k,\, k=1, \dots, m),
\end{equation}
where
\begin{equation}\label{defK2dhns}
K:=4(1+h)\nu^{-3}(1+{\blue\frac{\sqrt{hL}}{\sqrt2\pi}})(\|\sigma\|_H^2+\|\d_{z_2} \sigma\|_H^2).
\end{equation}
Then the stochastic hydrostatic Navier-Stokes equation \eqref{2dhns} has a unique invariant measure $\pi$. Further, there exist constants $C>0$, $r>0$ such that
\begin{equation*}
W_{\|\cdot\|_H\wedge1}(\Law (u^x(t)),\pi) \le C (1+\|x\|_H^2+\|\d_{z_2} x\|_H^2)e^{-rt },\quad t\ge0,\, x\in V.
\end{equation*}

\end{Theorem}
\begin{proof}
As in the proof of Theorem~\ref{tNS}, we will utilize Theorem~\ref{T:lyapSPDE}.  Let us check that all the conditions of Theorem~\ref{T:lyapSPDE} are satisfied.  Put
$$
U(x):=\|x\|_H^2+\|\d_{z_2} x\|_H^2, \quad S(x)=\|x\|_V^2+\|\d_{z_2} x\|_V^2, \quad q(x,y)=\|x-y\|_H^2, \quad x,y\in V.
$$
We use the same generalized coupling construction as in \cite[Section~6.2.2]{KS17} (which in turn is a minor variation of the construction from \cite[Section~3.2.4]{GMR17}).

Fix $x,y\in V$. Take $X^{x,y}:=u^x$ and let $Y^{x,y}$  be the solution to the same equation with the additional control  term in the right hand side
$$
\frac12\nu K P^K(X^{x,y}_t-Y^{x,y}_t)dt
$$
and starting from $Y^{x,y}_0=y$; here $P^K$ is the projection to the space $\Span (e_{i,j}:\,\lambda_{i,j}\le K )$.

Then an application of the It\^o formula, Poincar\'e inequality, {\blue and Ladyzhenskaia bound \cite[Formula (6)]{L59}} yields (see the inequality at \cite[line~1, p. 633]{GMR17})
\begin{multline}\label{diffest2dhns}
\frac12 d\|X^{x,y}(t)-Y^{x,y}(t)\|_H^2+\frac12\nu K\|X^{x,y}(t)-Y^{x,y}(t)\|_H^2dt\\
\le 4(1\vee h)\nu^{-1}\|X^{x,y}(t)\!-\!Y^{x,y}(t)\|_H^2 \bigl(\|X^{x,y}(t)\|_V^2\!+\!\|\d_{z_2} X^{x,y}(t)\|_H\|\d_{z_2} X^{x,y}(t)\|_V\bigr).
\end{multline}
By the Poincar\'e inequality, for any $\phi\in H^1(D)$ such that $\phi\evalat{\d D}=0$ we have
$$
{\blue \pi^2(h^{-2}+L^{-2})\|\phi\|_{L^2}\le\|\grad\phi\|_{L^2}.}
$$
This combined with \eqref{diffest2dhns} and the Gronwall inequality implies now
\begin{equation*}
q(X^{x,y}(t),Y^{x,y}(t))\le q(x,y) \exp\Bigl(-\nu Kt+\kappa\int_0^t (\|X^{x,y}(s)\|_V^2+\|\d_{z_2} X^{x,y}(s)\|_V^2)\,ds\Bigr),
\end{equation*}
where
$$
\kappa:=8(1\vee h)\nu^{-1}(1+{\blue\frac{\sqrt{hL}}{\sqrt2\pi}}).
$$
Thus Assumption~\textbf{H1} is satisfied.

To check \textbf{H2} we use the standard energy estimates  \cite[formulae (3.15) and (3.17)]{GMR17}. We get
\begin{multline}\label{enest2dhns}
d (\|X^{x,y}(t)\|_H^2+ \|\d_{z_2}X^{x,y}(t)\|_H^2)+2\nu(\|X^{x,y}(t)\|_V^2+\|\d_{z_2} X^{x,y}(t)\|_V^2)dt\\=(\|\sigma\|_H^2+\|\d_{z_2} \sigma\|_H^2)dt+M(t),
\end{multline}
where $M$ is a continuous local martingale with the quadratic variation
$$
d\langle M\rangle_t\leq\! (4 \|X^{x,y}(t)\|_H^2\|\sigma\|_H^2+4 \|\d_{z_2} X^{x,y}(t)\|_H^2\|\d_{z_2}\sigma\|_H^2)dt\le C\| X^{x,y}(t)\|_V^2dt\le \!
C S(X^{x,y}(t)),
$$
for some $C>0$; here in the second inequality we used the  Poincar\'e inequality. Thus \textbf{H2} is satisfied with $\mu:=2\nu$ and $b:=\|\sigma\|_H^2+\|\d_{z_2} \sigma\|_H^2$. Inequality \eqref{condtheta} holds due to the definition of $K$ in \eqref{defK2dhns}.

Denote
$$
W^{x,y}(t):=W(t)+\int_0^t \frac12\nu K\sigma^{-1} P^K(X^{x,y}_t-Y^{x,y}_t)dt,\quad t\ge0.
$$
By \eqref{nondeg2dhns}, we see that this process is well--defined. Arguing as in the proof of theorem  Theorem~\ref{tNS}, we see that the uniqueness of the solution to \eqref{2dhns} implies that for each $t\ge0$ there exists a measurable mapping $\Phi_t^y\colon \C([0,t],\R^m)\to V$ such that
$Y^{x,y}(t)=\Phi_t^y(W_{[0,t]}^{x,y})$ and $u^y(t)=\Phi_t^y(W_{[0,t]})$. Further, thanks again to condition \eqref{nondeg2dhns}, we see that (similar to \eqref{beta_XY}) there exists a constant $C>0$ such that for any $t\ge0$
$$
\|\frac12\nu K\sigma^{-1} P^K(X^{x,y}_t-Y^{x,y}_t)\|_{\R^m}^2\le C \|X^{x,y}_t-Y^{x,y}_t\|^2_H=C q(X^{x,y}_t,Y^{x,y}_t).
$$
Thus, by  Lemma~\ref{L:pomoshH3}, condition \textbf{H3} is also satisfied.

Finally, we consider the following Lyapunov function:
$$
V(x)=\|x\|_H^2+\|\d_{z_2} x\|_H^2,\quad x\in V.
$$
Condition \eqref{Llyap} follows from the energy estimate \eqref{enest2dhns}. By definition, for any $M>0$ the functions $U$ and $q$ are bounded on the level sets $\{V\le M\}$ and $\{V\le M\}\times\{V\le M\}$, respectively.

We note again that $\rho:=\|x-y\|_H\le q(x,y)^{1/2}$. Therefore all the conditions of Theorem~\ref{T:lyapSPDE} hold. This immediately implies the statement of the theorem.
\end{proof}

\subsection{The Fractionally Dissipative Euler Model}
In this and the next sections we will denote the fractional Laplacian by $\Lambda^\gamma:=-(-\Delta)^{\gamma/2}$, $\gamma\in(0,2)$. For a vector $\mathbf{f}(z)=(f_1(z),f_2(z))$, where $z=(z_1,z_2)\in\R^2$  we denote as usual $\curl \mathbf{f}:=\d_{z_1} f_2-\d_{z_2}f_1$

Fix $r>2$ and consider the stochastic fractionally dissipative Euler model on the periodic box $D:=[-\pi,\pi]^2$. In the velocity formulation the model is given by
\begin{align}\label{FDEM}
&\begin{aligned}[b]
d\u(z,t)+(\u(z,t)\cdot  \nabla)\u(z,t)dt&=(\Lambda^\gamma \u(z,t) -\nabla p(z,t))\,dt\\
&\phantom{=}+\sum_{k=1}^m \bsigma_k(z)dW^k(t),\quad z\in \D,\,t\ge0;
\end{aligned}\\
&\u(\cdot,0)=\bx,\,\,\nabla\cdot \u=0,\nn
\end{align}
 where
  $\mathbf{u} = (u_1, u_2)$ is the unknown velocity field, $p$ is the unknown pressure, $m\in\mathbb{N}$, $W=(W^1, W^2, \hdots, W^m)$ is a standard $m$--dimensional Brownian motion. We assume that for any $k=1,\hdots,m$ we have  $\bsigma_k\in H^{r+3}(\D)^2$, $\nabla\cdot \sigma_k=0$, $\int_D \sigma_k(z)\,dz=0$.

Introduce the following space:
 $$
 V:=\{\boldsymbol{\phi}\in H^r(D)^2\colon \nabla\cdot \boldsymbol{\phi}=0,\, \int_D \boldsymbol{\phi}(z)\,dz=0\}.
 $$
The eigenvectors $e_1, e_2,\hdots$ of the operator ${-\Lambda^\gamma}\bigr|_V$ (with the corresponding eigenvalues
$0<\lambda_1=1\le \lambda_2\le\hdots$) form a complete orthonormal basis of $V$.

 It is known (\cite[Section~3.3.1]{GMR17}, see also \cite[{\blue footnote at page 821}, Proposition~1.1 and Theorem~1.4]{CGHV14}) that for any $\bx\in V$ equation \eqref{FDEM} has a unique strong solution, which in the case of ambiguity will be further denoted by $\u^\bx$. Moreover, $\u$ is a Feller Markov process with the state space $V$. The existence and uniqueness of the invariant measure of $\u$ was shown in {\blue\cite[ Theorem 1.5]{CGHV14} and} \cite[Section~3.3]{GMR17}. The weak convergence of transition probabilities to the invariant measure was established in \cite[Section~6.2.3]{KS17}.
 It was conjectured in \cite[Section~3.3]{GMR17} that to establish the rate of convergence to the invariant measure one might need to exploit  high-order polynomial moment bounds of the solution in high-order Sobolev spaces. Actually, applying our method, we are able to show exponential rate of convergence without invoking these complicated bounds; we use just the calculations from \cite[Section~3.3.2]{GMR17}.


 \begin{Theorem}  There exists a universal $C>0$ such that if for some $N\in\Z_+$ one has
$$
\Span (e_k,\, k=1, \dots, N)\subset\Span (\bsigma_k,\, k=1, \dots, m)
$$
and
\begin{equation}\label{granicafrSNS}
\lambda_{N+1}>C \|\bsigma\|^2_{L_{\blue 6/\gamma}},
\end{equation}
then the stochastic fractionally dissipative Euler equation \eqref{FDEM} has a unique invariant measure $\pi$. Further,  there exist constants $C_1>0$, $r>0$ such that
\begin{equation*}
W_{\|\cdot\|_{L_2}\wedge1}(\Law (\u^\bx(t)),\pi) \le C_1 (1+\|\curl \bx\|^2_{L_{6/\gamma}}+\|\bx\|^2_{L_2})e^{-rt },\quad t\ge0,\,\bx\in V.
\end{equation*}
\end{Theorem}
\begin{proof}
Let us again check that all the conditions of Theorem~\ref{T:lyapSPDE} are satisfied. For brevity we denote $p:=6/\gamma$  and put
$$
U(\bx):=(1+\|\curl \bx\|_{L_p}^p)^{2/p}, \quad S(\bx):=\|\curl \bx\|_{L_p}^2, \quad q(\bx,\by):=\|\bx-\by\|^2_{L_2},\quad \bx,\by\in V.
$$
Since $\bx\in H^r(D)^2$, we see that $\curl \bx\in H^{r-1}(D)$ and thus by the Sobolev embedding theorem $\curl \bx\in L_\infty(D)$ thanks to the assumption $r>2$. Therefore $U$ is finite on $V$.

We again use the same coupling construction as in \cite[Section~6.2.3]{KS17} (which in turn is a small modification of the construction from \cite[Section~3.3.2]{GMR17}). Fix $\bx,\by\in V$. We take $\bX^{\bx,\by}:=\u^\bx$ and  let  $\bY^{\bx,\by}$ be the solution to the same  equation with the extra term in the right--hand side
$$
\frac12\lambda_{N+1} P_N(\bX^{\bx,\by}(t)-\bY^{\bx,\by}(t))dt
$$
and started from the initial condition $\bY^{\bx,\by}_0=\by$.  Then by \cite[formula (3.28)]{GMR17} (see also  \cite[formulae (4.15)--(4.20)]{CGHV14}) there exists a universal $C_1>0$ such that
$$
q(\bX^{\bx,\by}(t),\bY^{\bx,\by}(t))\le q(\bx,\by)\exp\bigl(-\lambda_{N+1}t+C_1\int_0^t S(\bX^{\bx,\by}(s))\,ds\bigr).
$$
Therefore assumption \textbf{H1} is satisfied.

Assumption \textbf{H2} follows from the corresponding energy estimate \cite[formula (3.29)]{GMR17}. Indeed, this estimate immediately implies that for some universal constants $C_2,C_3>0$ we have
$$
d U(\bX^{\bx,\by}(t))+C_2 S(\bX^{\bx,\by}(t))dt\le C_3 \|\sigma\|_{L^p}^2dt +M(t),
$$
where $M$ is a continuous local martingale with quadratic variation bounded by
$$
d\langle M\rangle_t\leq 4 \|\sigma\|_{L^p}^2 U(X^{x,y}(t))dt\le 4 \|\sigma\|_{L^p}^2 (1+S(X^{x,y}(t)))dt,
$$
see \cite[formula (3.30)]{GMR17}. This implies \textbf{H2}. Inequality  \eqref{condtheta} follows from \eqref{granicafrSNS}, where we took $C:=C_1C_3/C_2$.

Assumption \textbf{H3} follows from the uniqueness of a strong solution to \eqref{FDEM} by exactly the same argument as in the proofs of Theorem~\ref{tNS} and \ref{T:2dhns}.

Finally we consider the Lyapunov function
$$
V(\bx):=U(\bx)+\|\bx\|_{L_2}^2,\quad  \bx\in V.
$$
It is clear that for any $M>0$ the functions $U$ and $q$ are bounded on the level sets $\{V\le M\}$ and $\{V\le M\}\times\{V\le M\}$, respectively. Further, we have a simple energy estimate
\begin{equation}\label{Ito2dfsns}
d \|\bX^{\bx,\by}(t)\|^2_{L^2}+2\|\Lambda^{\gamma/2}\bX^{\bx,\by}(t)\|^2_{L^2} dt=\|\sigma\|^2_{L^2}\, dt+d\wt M(t),
\end{equation}
for some martingale $\wt M$. Since $\bX^{\bx,\by}(t)$ is a mean zero function, the Poincar\'e inequality implies
\begin{equation*}
\|\Lambda^{\gamma/2}\bX^{\bx,\by}(t)\|_{L^2}\ge\|\bX^{\bx,\by}(t)\|_{L^2}.
\end{equation*}
Combining this with \eqref{Ito2dfsns} and with the energy estimate for $\|\curl \bx\|_{L_p}^p$ \cite[formula (3.29)]{GMR17}, we obtain \eqref{Llyap}.

We note again that $\rho:=\|\bx-\by\|_{L_2}\le q(\bx,\by)^{1/2}$. Therefore all the conditions of Theorem~\ref{T:lyapSPDE} hold. This immediately implies the statement of the theorem.
\end{proof}

\subsection{The 2D Damped Stochastically Forced Euler-Voigt Model}

Our next example is the 2D damped stochastically forced Euler--Voigt Model. Fix $\gamma>2/3$  and consider the following equation on the periodic box $D:=[-\pi,\pi]^2$
\begin{align}\label{neg2dsns}
&d\u(z,t)\!+\!(\nu\u(z,t)\!+\! (\u_\gamma(z,t)\cdot  \nabla)\u_\gamma(z,t)dt\!+\!\nabla p(z,t))dt=\sum_{k=1}^m \bsigma_k(z)dW^k(t),\\
&\u(\cdot,0)=\bx,\,\,\nabla\cdot \u=0,\,\, \Lambda^\gamma\u_\gamma=\u,\nn
\end{align}
 where  $\mathbf{u} = (u_1, u_2)$ is the unknown velocity field, $p$ is the unknown pressure, $\nu>0$, $m\in\mathbb{N}$, $W=(W^1, W^2, \hdots, W^m)$ is a standard $m$--dimensional Brownian motion. We assume that for any $k=1,\hdots,m$ we have  $\bsigma_k\in H^{1-\gamma/2}(\D)^2$, $\nabla\cdot \sigma_k=0$, $\int_D \sigma_k(z)\,dz=0$.

Introduce the following space:
 $$
 V:=\{\boldsymbol{\phi}\in H^{1-\gamma/2}(D)^2\colon \nabla\cdot \boldsymbol{\phi}=0,\, \int_D \boldsymbol{\phi}(z)\,dz=0\}.
 $$
The eigenvectors $e_1, e_2,\hdots$ of the operator $-\Delta\bigr|_V$ (with the corresponding eigenvalues
$0<\lambda_1=1\le \lambda_2\le\hdots$) form a complete orthonormal basis of $V$.

 It is known (\cite[Proposition~3.4]{GMR17}) that for any $\bx\in V$ equation \eqref{FDEM} has a unique strong solution, which in the case of ambiguity will be further denoted by $\u^\bx$. Moreover, $\u$ is a Markov process with the state space $V$. Note however that the Feller property of $\u$ is known only with respect to a weaker $H^{-\gamma/2}$ norm rather than  $H^{1-\gamma/2}$ norm. This causes several technical difficulties, which however are not very crucial and can be successfully resolved.

 First, we note that the process $\u^\bx$ is progressively measurable. This follows from the facts that  $\u^\bx$ is stochastically continuous in the weaker $H^{-\gamma/2}$ norm and the Borel $\sigma$-algebras generated by the $H^{-\gamma/2}$- and the $H^{1-\gamma/2}$-norms are the same.

  Next, we cannot apply here Theorem~\ref{T:lyapSPDE}, which requires the Feller property of the Markov semigroup.
 Therefore to show the exponential convergence of transition probabilities we apply  \cite[Theorem~4.8]{HMS11}, which does not require the Feller property, together with Lemma~\ref{P:41} and Theorem~\ref{L:contr}. Note that the existence of the  invariant measure of $\u$ does not follow from these considerations; fortunately, it was already established in
 \cite[Proposition~3.5]{GMR17}.

In this section we also use the notation $\|\phi\|_{H^s}:=\|\Lambda^s \phi\|_{L^2}$, $s\in\R$.

 \begin{Theorem}  There exists a universal $C>0$ such that if for some $N\in\Z_+$ one has
$$
\Span (e_k,\, k=1, \dots, N)\subset\Span (\bsigma_k,\, k=1, \dots, m)
$$
and
\begin{equation}\label{granicaEVM}
\lambda_{N+1}^{\gamma/2-1/3}>C \frac{\|\curl \bsigma\|^2_{H^{-\gamma/2}}}{\nu^3},
\end{equation}
then  equation \eqref{neg2dsns} has a unique invariant measure $\pi$. Further, there exist constants $C_1>0$, $r>0$ such that
\begin{equation}\label{rateneg2dsns}
W_{\|\cdot\|_{H^{-\gamma/2}}\wedge1}(\Law (\u^\bx(t)),\pi) \le C_1 (1+\|\bx\|^2_{H_{1-\gamma/2}})e^{-rt },\quad t\ge0,\,\bx\in V.
\end{equation}
\end{Theorem}
\begin{proof}
First let us check that Assumptions \textbf{H1}--\textbf{H3} hold.

Put
$$
U(\bx)=S(\bx):=\|\Lambda^{-\gamma/2}\curl \bx\|_{L_2}^2, \quad q(\bx,\by)=\|\Lambda^{-\gamma/2}(\bx-\by)\|_{L_2}^2,\quad \bx,\by\in V.
$$

We use the same coupling construction as in \cite[Section~6.2.4]{KS17} (which again is a minor modification of the corresponding construction from \cite[Section~3.4.4]{GMR17}). Fix $\bx,\by\in V$. Take $\bX^{\bx,\by}:=\u^\bx$ and  let  $\bY^{\bx,\by}$ be the solution to the same  equation with the extra term in the right--hand side
$$
\frac12\nu \lambda_{N+1}^{\gamma/2-1/3} P_N(\bX^{\bx,\by}(t)-\bY^{\bx,\by}(t))dt
$$
and started from the initial condition $\bY^{\bx,\by}_0=\by$ (recall that $\gamma>2/3$).  Then \cite[formula (3.45)]{GMR17} and the Sobolev embedding $H^{1/3}\subset L_3$ imply that for some universal constant $C_1>0$
\begin{align}\label{Chastodinotsenki}
\frac12\frac{d}{dt}&q(\bX^{\bx,\by}(t),\bY^{\bx,\by}(t))\!+\!\nu q(\bX^{\bx,\by}(t),\bY^{\bx,\by}(t))\!+\!
\frac12\nu \lambda_{N+1}^{\gamma/2-1/3} \|P_N\Lambda^{-\gamma/2}(\bX^{\bx,\by}(t)\!-\!\bY^{\bx,\by}(t))\|_{L_2}^2\nn\\
&\le C_1\|\Lambda^{-\gamma}(\bX^{\bx,\by}(t)-\bY^{\bx,\by}(t))\|^2_{H^{1/3}}\|\Lambda^{-\gamma}\curl(\bX^{\bx,\by}(t))\|_{H^{1/3}}.
\end{align}
By the Poincar\'e inequality,
\begin{align*}
C_1\|\bX^{\bx,\by}&(t)-\bY^{\bx,\by}(t)\|^2_{H^{-\gamma+1/3}}\|\curl(\bX^{\bx,\by}(t))\|_{H^{-\gamma+1/3}}\\
\le&\frac12
\|\bX^{\bx,\by}(t)-\bY^{\bx,\by}(t)\|^2_{H^{-\gamma+1/3}}\Bigl(\nu \lambda_{N+1}^{\gamma/2-1/3}\!+\nu^{-1} \lambda_{N+1}^{-\gamma/2+1/3} C_1^2\|\curl(\bX^{\bx,\by}(t))\|_{H^{-\gamma+1/3}}^2\Bigr)\\
\le&\frac12\nu \lambda_{N+1}^{\gamma/2-1/3} \|P_N\Lambda^{-\gamma/2}(\bX^{\bx,\by}(t)\!-\!\bY^{\bx,\by}(t))\|_{L_2}^2
+\frac12\nu q(\bX^{\bx,\by}(t),\bY^{\bx,\by}(t))\\
&+\frac12\nu^{-1} \lambda_{N+1}^{-\gamma/2+1/3}C_1^2 q(\bX^{\bx,\by}(t),\bY^{\bx,\by}(t))S(\bX^{\bx,\by}(t)).
\end{align*}
Combining this with \eqref{Chastodinotsenki} and applying the Gronwall inequality, we obtain
$$
q(\bX^{\bx,\by}(t),\bY^{\bx,\by}(t))\le q(\bx,\by)\exp\bigl(-\nu t+C_1^2\nu^{-1}\lambda_{N+1}^{-\gamma/2+1/3}\int_0^t S(\bX^{\bx,\by}(s))\,ds\bigr).
$$
Therefore assumption \textbf{H1} is satisfied.

Assumption \textbf{H2} follows immediately from \cite[formula (3.35)]{GMR17}. Indeed, rewriting this formula in our notation we get,
$$
d U(\bX^{\bx,\by}(t))+2\nu S(\bX^{\bx,\by}(t)) dt=\|\curl \bsigma\|_{H^{-\gamma/2}}^2dt+M(t),
$$
where $M$ is a continuous local martingale with quadratic variation bounded by
$$
d\langle M\rangle_t\leq 4 \|\curl \bsigma\|_{H^{-\gamma/2}}^2 S(\bX^{\bx,\by}(t))dt.
$$
This yields \textbf{H2}. Inequality  \eqref{condtheta} follows from \eqref{granicaEVM}, where we took $C:=C_1^2/2$.

To verify \textbf{H3}, we use the same argument as in the proofs of Theorems~\ref{tNS} and \ref{T:2dhns}. Note that condition 3 of Lemma~4.1 holds since by the Poincar\'e inequality for any $t\ge0$
$$
\|P_N(\bX^{\bx,\by}(t)-\bY^{\bx,\by}(t))\|_{L_2}^2\le\lambda_N^{\gamma/2}\|P_N(\bX^{\bx,\by}(t)-\bY^{\bx,\by}(t))\|_{H^{-\gamma/2}}^2\le\lambda_N^{\gamma/2} q(\bX^{\bx,\by}(t),\bY^{\bx,\by}(t)).
$$

Thus all the assumptions \textbf{H1}--\textbf{H3} hold.

Finally, we consider the Lyapunov function
$$
V(\bx):=\|\curl \bx\|_{H^{-\gamma/2}}^2+\|\bx\|_{H^{-\gamma/2}}^2,\quad  \bx\in V.
$$
Obviously, for any $M>0$ the functions $U$ and $q$ are bounded on the level sets $\{V\le M\}$ and $\{V\le M\}\times\{V\le M\}$, respectively. Further, by adding  energy estimates \cite[formulae (3.34) and (3.35)]{GMR17}, we obtain
$$
\E V(\u^\bx(t))\le  V(\bx) e^{-2\nu t}+K_V,\quad t\ge0
$$
for some universal constant $K_V>0$.

Now we are ready to complete the proof of the theorem. The existence of an invariant measure $\pi$ was established in
 \cite[Proposition~3.5]{GMR17}. Since $V$ is a Lyapunov function for $P_t$, $V$ is integrable with respect to  any invariant measure (see, e.g, \cite[Proposition~4.24]{HaiCon}). Applying Lemma~\ref{P:41} and Theorem~\ref{L:contr}(i),~(iii), we see that for some $N>0$, $t>0$, $\alpha>0$, $C>0$ the distance-like function
$$
d_N(\mathbf{x_1},\mathbf{x_2}):=\bigl(N e^{C (U(\mathbf{x_1})\wedge U(\mathbf{x_2}))}q(\mathbf{x_1},\mathbf{x_2})^\alpha\bigr)\wedge 1
$$
is contracting for $P_t$ and the set $\{V\le 4 K_V\}$ is $d_N$--small for $P_t$. Thus all the conditions of \cite[Theorem~4.8]{HMS11} are satisfied. Hence the invariant measure is unique and the exponential bound on the convergence rate \eqref{rateneg2dsns} follows from \cite[formula (4.6)]{HMS11} and the integrability of $V$ with respect to $\pi$. Note that to use the integrability, one first has to check that $W_d(P_t(x, \cdot), P_t(y, \cdot))$ is measurable. This is the point, where the lack of the Feller property causes minor technical difficulty; in particular we can not refer here neither to  \cite[Lemma~{\blue 4.13}]{HMS11} nor to \cite[Theorem~4.4.3]{K17}. Note however, that the mapping
$$
V\times V\ni (x,y)\mapsto (P_t(x, \cdot), P_t(y, \cdot))\in \mathcal{P}(V)\times \mathcal{P}(V)
$$
is measurable (one can get this easily using the $H^{-\gamma/2}$-Feller property), and the function $d_N$ is continuous on $V\times V$. Then the required measurability follows by \cite[Corollary~5.22]{Villani}.
\end{proof}

\subsection{Boussinesq approximation for Rayleigh--B\'enard convection}
Finally, we consider Boussinesq approximation for the Rayleigh--B\'enard convection perturbed by additive noise. The physical motivation behind the model as well as the relevance of the model for fluid dynamics are explained in detail in \cite{FGRW16}.

Consider the following system of equations evolving on a domain $D:=[0,L]\times[0,1]$
\begin{align}\label{Bous1}
&\begin{aligned}[b]
\frac{1}{Pr}(d\u(z,t)+(\u(z,t)\cdot  \nabla)\u(z,t)dt)&=\Delta\u(z,t)dt +Ra\hspace{0.1em} \mathbf{i}_2\hspace{0.1em} T (z,t)dt -\nabla p(z,t)dt\\
&\phantom{=}+\sum_{k=1}^{m_1} \bsigma_k(z)dW^k(t),\quad z\in \D,\,t\ge0;
\end{aligned}\\
&d T(z,t) + (\u(z,t)\cdot  \nabla) T(z,t)dt=\Delta T(z,t)dt +\sum_{k=1}^{m_2} \rho_k(z)dB^k(t),\quad z\in \D,\,t\ge0;\label{Bous2}\\
&(\u(0), T(0))=\bx,\,\,\nabla\cdot \u=0,\nn\\
&\u\evalat{z_2=0}=\u\evalat{z_2=1}=0,\quad T\evalat{z_2=0}=\wt {Ra},\quad T\evalat{z_2=1}=0,\quad \text{$\u, T$ are periodic in $z_1$},\nn
\end{align}
 where $\mathbf{u} = (u_1, u_2)$ is the unknown velocity field; $p$ is the unknown pressure; $T$ is the unknown temperature;  $m_1,m_2\in\Z_+$; $W=(W^1, \hdots, W^{m_1})$ and $B=(B^1, \hdots, B^{m_2})$ are standard independent  $m_1$-- and $m_2$--dimensional Brownian motions, respectively. $Pr$, $Ra$, $\wt{Ra}$ denote positive constants and correspond to some physical parameters: $Pr$ is the Prandtl number, $Ra$ and $\wt{Ra}$ are Rayleigh numbers. We also used the notation  $\mathbf{i}_2:=(0,1)$.

Introduce the following spaces:
 \begin{align*}
 &V_1:=\{\boldsymbol{\phi}\in H^1(D)^2\colon \nabla\cdot \boldsymbol{\phi}=0,\, \boldsymbol{\phi}\evalat{z_2=0,1}=0, \text{$\boldsymbol{\phi}$ is periodic in $z_1$}\},\\
 &V_2:=\{\phi\in H^1(D)\colon \phi\evalat{z_2=0,1}=0,\, \text{$\phi$ is periodic in $z_1$}\},\\
 &\wt V_2:=\{\phi\in H^1(D)\colon \phi\evalat{z_2=0}=\wt {Ra}, \phi\evalat{z_2=1}=0,\, \text{$\phi$ is periodic in $z_1$}\}.
 \end{align*}
We denote by $H_1$ and $H_2$ the completions of $V_1$ and $V_2$ with respect to $L^2(D)^2$-- and $L^2(D)$--norms, correspondingly.  Put $V:=V_1\times V_2$, $H:=H_1\times H_2$.   The eigenvectors $e_1, e_2,\hdots$ of the Stokes operator (with the corresponding eigenvalues $0<\lambda_1\le \lambda_2\le\hdots$) form a complete orthonormal basis of $H_1$, and the eigenvectors $j_1, j_2,\hdots$ of the Laplace operator  (with the corresponding eigenvalues $0<\mu_1\le \mu_2\le\hdots$)  form a complete orthonormal basis of $H_2$. For $N\in\Z_+$ we denote by $P^{H_i}_N$ the projection onto the span of the first $N$ eigenvectors of $H_i$, $i=1,2$. By $\|\cdot\|_{H_1}$, $\|\cdot\|_{H_2}$, $\|\cdot\|_{H}$ we denote the $L_2$ norms in the spaces $H_1$, $H_2$, $H$, correspondingly.

 We assume that for each $k=1,\hdots, m_1$ we have $\bsigma_k\in H^{2}(\D)^2\cap V_1$ and for each
  $k=1,\hdots, m_2$ we have $\rho_k\in H^{2}(\D)\cap V_2$.

 It is known (\cite[Proposition~2.1]{FGRW16}) that for any $\bx\in H$ a system of equations \eqref{Bous1}--\eqref{Bous2} has a unique strong solution, which in the case of ambiguity will be further denoted by $(\u^\bx, T^\bx)$. Moreover, $(\u, T)$ is a Feller Markov process with the state space $H$ and for almost all $t>0$ we have $(\u^\bx(t), T^\bx(t))\in V_1\times \wt V_2$.

 It was shown in \cite[Theorem~1.1]{FGRT15} that if the velocity $u$ and temperature $T$ satisfies periodic boundary conditions, then the process  $(\u, T)$ is exponentially ergodic. However, as was noted in \cite{FGRW16}, ``using periodic boundary conditions in the vertical directions is not appropriate from the physical point of view''. Therefore, following \cite{FGRW16}, we equip the system  \eqref{Bous1}--\eqref{Bous2} with mixed periodic and non-homogeneous Dirichlet boundary conditions that are physically relevant.

 Note that the methods of  \cite[Theorem~1.1]{FGRT15} are not applicable to treat the system with mixed boundary conditions, see also a related discussion in \cite[Page 619, lines~36--42]{GMR17}. However, using Theorem~\ref{T:lyapSPDE}, we are able to establish exponential ergodicity of $(\u^\bx, T^\bx)$ with mixed boundary conditions in the case of small Rayleigh numbers

 \begin{Theorem}\label{T:Bous} Suppose that
 $$
 \wt{Ra} Ra < {\blue \pi^2\sqrt2-1}.
 $$
Assume further that there exists $N_1, N_2 \in\Z_+$ such that
$$
\P_{N_1}^{H_1} H_1 \subset\Span (\bsigma_k,\, k=1, \dots, m_1); \quad \P_{N_2}^{H_2} H_2 \subset\Span (\rho_k,\, k=1, \dots, m_2);
$$
and
\begin{align}
&\lambda_{N_1+1}>C_1(Pr)^{-3/2} (Pr  \|\bsigma\|^2_{H_1}+Ra^2\|\rho\|^2_{H_2}),\label{granicaBous}\\
&\mu_{N_2+1}>C_1(Pr)^{-1/2} (Pr  \|\bsigma\|^2_{H_1}+Ra^2\|\rho\|^2_{H_2})+C_2 Pr,\label{granicaBous2}
\end{align}
where $C_1={\blue \frac{ (1+1/\pi)(\sqrt2+16Pr^{-1/2})}{1\wedge \bigl(Ra^2(2-\pi^{-4}(1+Ra\wt{Ra})^2)\bigr)}}$, $C_2={\blue 2\pi^{-2}}(Ra+\wt{Ra}/Pr)^2$.

Then the system \eqref{Bous1}--\eqref{Bous2} has a unique invariant measure $\pi$. Further,  there exist constants $C_3>0$, $r>0$ such that
\begin{equation*}
W_{\|\cdot\|_{H}\wedge1}(\Law (\u^\bx(t),T^\bx(t)),\pi) \le C_3 (1+\| \bx\|^2_{H})e^{-rt },\quad t\ge0,\,\bx\in H.
\end{equation*}
\end{Theorem}

\begin{Remark}
Note that conditions \eqref{granicaBous}--\eqref{granicaBous2} are satisfied in the following two different scenarios. The first case is when the noise acts both on the velocity $\u$ and temperature $T$ (i.e, both $N_1$ and $N_2$ are large enough). The second case is when the  noise acts only on temperature (i.e., $N_2$ is large enough). Then the system is ergodic provided that the Prandtl number is large enough (and $N_1$ can be small, in particular equal to $0$).
\end{Remark}

\begin{proof}[Proof of Theorem~\ref{T:Bous}]
First, as in \cite{FGRW16}, we introduce a shifted temperature process
$$
\Theta(t,z):=T(t,z)-\wt {Ra} (1-z_2),\quad t\ge0,\,\, z\in D.
$$
It is clear from the definition that $\Theta$ is periodic in $z_1$ and satisfies homogeneous boundary conditions $\Theta\evalat{z_2=0}=\Theta\evalat{z_2=1}=0$.

Let us check that all the conditions of Theorem~\ref{T:lyapSPDE} are satisfied for the Markov process $(\u, \Theta)$. Since $\Theta$ is just $T$ shifted by some non-random function, which is constant in time, exponential ergodicity of the Markov process  $(\u, \Theta)$ would imply exponential ergodicity of $(\u, T)$.

Take
$$
U(\bx):=\frac{1}{Pr}\|(x_1,x_2)\|_{H_1}^2+Ra^2\|x_3\|_{H_2}^2,\quad\! S(\bx):=\|\grad \bx\|, \quad\! q(\bx,\by):=\|\bx-\by\|_{H},\quad \bx,\by\in H.
$$
We use the coupling construction, which is a small modification of the construction from \cite[Section~5.2]{FGRW16}. Fix $\bx,\by\in H$. We take $\bX^{\bx,\by}:=(\u^\bx,\Theta^\bx)$ and  let  $\bY^{\bx,\by}:=(\wt \u,\wt \Theta)$ be the solution to the same system of  equations with the initial condition $\by$ and the additional control terms:
\begin{align*}
&\frac{1}{Pr}(d\wt \u+(\wt\u\cdot  \nabla)\wt\u dt)=\Delta\wt\u dt +Ra\hspace{0.1em} \mathbf{i}_2\hspace{0.1em} \wt \Theta dt -\nabla \wt p dt+\frac12\lambda_{N_1+1}P^{H_1}_{N_1}(\u^{\blue \bx}-\wt \u)dt+\sum_{k=1}^{m_1} \bsigma_kdW^k,\\
&d \wt\Theta + (\wt\u\cdot  \nabla) \wt\Theta dt=\wt {Ra} \wt u_2dt + \Delta \wt\Theta dt +
\frac12\mu_{N_2+1}P^{H_2}_{N_2}(\Theta^{\blue \bx}-\wt \Theta)dt+
\sum_{k=1}^{m_2} \rho_kdB^k,\\
&(\wt \u(0), \wt\Theta(0))=\by,\,\,\nabla\cdot \wt \u=0,\nn\\
&\wt \u\evalat{z_2=0}=\wt\u\evalat{z_2=1}=0,\quad \wt \Theta \evalat{z_2=0}=\wt \Theta\evalat{z_2=1}=0,\quad \text{$\wt\u, \wt \Theta$ are periodic in $z_1$}.
\end{align*}

Then it follows from \cite[formulae (5.9)--(5.12)]{FGRW16} that
\begin{align}\label{Bousdiff}
&\frac12\frac{d}{dt} (\|\u^\bx-\wt \u\|_{H_1}^2+\|\Theta^\bx-\wt \Theta\|_{H_2}^2) +\frac{Pr}2\lambda_{N_1+1}
\|\u^\bx-\wt \u\|_{H_1}^2+\frac12\mu_{N_2+1} \|\Theta^\bx-\wt \Theta\|_{H_2}^2\nn\\
&\qquad \le \frac{(\wt {Ra} + Pr Ra)^2}{{\blue\pi^2}Pr}\|\Theta^\bx-\wt \Theta\|_{H_2}^2\nn\\
&\phantom{\qquad\le}+\frac{ C_4}{\sqrt {Pr}}  (\|\u^\bx-\wt \u\|_{H_1}^2+\|\Theta^\bx-\wt \Theta\|_{H_2}^2)( \|\nabla\u^\bx\|_{H_1}^2+\|\nabla\Theta^\bx\|_{H_2}^2),
\end{align}
where $C_4\!:={\blue (1+1/\pi)(1/\sqrt2+8/\sqrt{Pr})}$. Here we have also used the Poincar\'e inequality  {\blue $\|\nabla \phi\|_{L_2}\ge\pi\| \phi\|_{L_2}$} and the Sobolev embedding {\blue $\|{\phi}\|_{L_4}^4\le 2(1+1/\pi)\|{\phi}\|_{L_2}^2\| \nabla{\phi}\|_{L_2}^2$}; both are valid for any $\phi\in V_2$.

Inequality \eqref{Bousdiff} and the Gronwall inequality imply for any $t\ge0$
\begin{equation*}
q(\bX^{\bx,\by}(t), \bY^{\bx,\by}(t))\le q(\bx,\by)\exp\bigl(-\zeta t +\frac{{\blue 2 C_4}}{\sqrt {Pr}}\int_0^t \|\nabla\u^\bx(s)\|_{H_1}^2+\|\nabla\Theta^\bx(s)\|_{H_2}^2\,ds  \bigr),
\end{equation*}
where we put $\zeta:=(Pr\lambda_{N_1+1})\wedge (\mu_{N_2+1}- \frac{2(\wt {Ra} + Pr Ra)^2}{{\blue\pi^2}Pr})$. Therefore assumption \textbf{H1} is satisfied.

To verify Assumption \textbf{H2} we use
%
the corresponding energy estimates for $\u^\bx$ and $\Theta^\bx$ (see  \cite[formulae (3.32) and (3.35)]{FGRW16}). We multiply the first of these estimates by $1/Pr$, the second one by $Ra^2$ and  get
\begin{align}\label{BouseeuTheta}
&\frac1{Pr}d\|\u^\bx(t) \|^2_{H_1}\!+\!Ra^2d\|\Theta^\bx(t) \|^2_{H_2}\!+\! \|\nabla\u^\bx(t)\|_{H_1}^2dt +Ra^2\bigl({\blue 2\!-\!\frac{(1+Ra\wt{Ra})^2}{\pi^4}} \bigr)\|\nabla\Theta^\bx(t)\|_{H_2}^2dt \nn\\
&\qquad\le (Pr  \|\bsigma\|^2_{H_1}+Ra^2\|\rho\|^2_{H_2}) dt + d M(t),
\end{align}
where  $M$ is a continuous local martingale with  $M(0)=0$; its quadratic variation is  bounded by
$$
d\langle M\rangle_t\leq 4  (\|\bsigma\|^2_{H_1} +Ra^{\blue4}\|\rho\|^2_{H_2})(\|\u^\bx(t) \|^2_{H_1}+\|\Theta^\bx(t) \|^2_{H_2})dt.
$$

%
Inequality \eqref{BouseeuTheta} yields
\begin{equation*}
d U(\bX^{\bx,\by}(t))+\mu S(\bX^{\bx,\by}(t))dt \le (Pr  \|\bsigma\|^2_{H_1}+Ra^2\|\rho\|^2_{H_2}) dt + d M(t),
\end{equation*}
where $\mu:={\blue 1\wedge \bigl(Ra^2(2-\frac{(1+Ra\wt{Ra})^2}{\pi^4})\bigr)}$; this implies \textbf{H2}. Inequality  \eqref{condtheta} follows from \eqref{granicaBous} and \eqref{granicaBous2}.

Assumption \textbf{H3} follows from the uniqueness of a strong solution to \eqref{Bous1}--\eqref{Bous2} by the same argument as in the proofs of Theorem~\ref{tNS} and \ref{T:2dhns}.

Finally we consider the Lyapunov function
$$
V(\bx):=U(\bx),\quad  \bx\in H.
$$
We immediately see by definition, that for any $M>0$ the functions $U$ and $q$ are bounded on the level sets $\{V\le M\}$ and $\{V\le M\}\times\{V\le M\}$, respectively. Inequality \eqref{Llyap} follows directly from the energy estimate \eqref{BouseeuTheta} and the Poincar\'e inequality.

Thus all the conditions of Theorem~\ref{T:lyapSPDE} hold. This immediately implies the statement of the theorem.
\end{proof}

\appendix
\section*{Appendix. \,\,Distance between the law of an  It\^o process and the Wiener measure}
\renewcommand{\theequation}{A.\arabic{equation}}
\setcounter{equation}{0}

\renewcommand{\theTheorem}{A.\arabic{Theorem}}
\setcounter{Theorem}{0}
In the Appendix we provide useful bounds on the various distances between  the law of an  It\^o process and the Wiener measure under different sets of conditions. We use these bounds throughout the paper, however we believe that they are also of independent interest.

We begin by recalling that for a pair of probability measures $\mu\ll \nu$ over a measurable space $(X,\mathcal{X})$  \emph{the  Kullback--Leibler $($KL--$)$ divergence of $\mu$ from $\nu$}  is defined by
$$
D_{KL}(\mu\|\nu):=\int_X \log {\frac{\dd \mu}{\dd \nu}}\, \dd \mu=\int_X \frac{\dd \mu}{ \dd \nu}\log\Bigl(\frac{\dd \mu}{\dd \nu}\Bigr)\, \dd \nu.
$$
If a measure $\mu$ is not absolutely continuous with respect to $\nu$,  then for convenience we  put ${D_{KL}(\mu\|\nu):=+\infty}$.  KL--divergence is a stronger measure of difference between probability distributions than the total variation distance; they are connected in the following way.

\begin{Lemma}\label{l_TV_via_KL} Let $\mu$ and $\nu$ be probability measures over $(X,\mathcal{X})$. Then
\begin{align}
\label{otzenka2}
&d_{TV}(\mu,\nu)\le \sqrt{{\frac12} D_{KL}(\mu\|\nu)};\\
\label{otzenka}
&d_{TV}(\mu,\nu)\le 1-\frac12 e^{-D_{KL}(\mu\|\nu)};\\
\label{log-Scheffe}
&\int_X \Bigl(\log {\frac{\dd \mu}{\dd \nu}}\Bigr)_+d \mu\le  D_{KL}(\mu\|\nu)+\log 2.
\end{align}
Further, for any $N>1$ and any set $A\in\mathcal{X}$,
\begin{equation}
\label{diffbound}
\nu(A)\ge\frac1N \mu(A)-\frac{D_{KL}(\mu\|\nu)+\log 2}{N\log N}.
\end{equation}

\end{Lemma}
\begin{proof} We can consider only the case when $\mu\ll \nu$; otherwise ${D_{KL}(\mu\|\nu):=+\infty}$ and all the bounds trivially hold. \eqref{otzenka2} is the classical Pinsker inequality (see, e.g., \cite[Lemma~2.5.(i)]{Ts}).  Inequality \eqref{otzenka} is \cite[formula~(2.25)]{Ts}. To establish \eqref{log-Scheffe} denote $f:=\frac{\dd \mu}{\dd \nu}$. Then
$$
\int_X(\log f)_+\dd \mu=D_{KL}(\mu\|\nu)+\int_X(\log f)_-\dd \mu
$$
and
\begin{align*}
\int_X(\log f)_-\dd \mu&\leq \log\Bigl( \int_X e^{(\log f)_-}\dd \mu\Bigr)=\log\Bigl( \int_X \max\big(f^{-1}, 1\big) f\,\dd \nu\Bigr)
\\&\leq \log\bigl( \int_X (f+1) \dd \nu\bigr)=\log 2.
\end{align*}
This implies \eqref{log-Scheffe}.

To prove \eqref{diffbound} we fix $N>1$ and a measurable set $A$. We have
\begin{align*}
\nu(A)&= \int_A \frac{d\nu}{d\mu}d\mu\nn\\
&\ge \frac1N\int_A \I(\frac{d\nu}{d\mu}\ge\frac1N)d\mu\nn\\
&\ge\frac1N \mu(A)-\frac1N\int_{X}\I(\frac{d\mu}{d\nu}\ge N)d\mu\nn\\
&\ge\frac1N \mu(A)-\frac1{N\log N}\int_{X}\log\bigl(\frac{d\mu}{d\nu}\bigr)_+d\mu.
\end{align*}
This combined with \eqref{log-Scheffe} yields \eqref{diffbound}.
\end{proof}

Consider now a $d$-dimensional ($d\in\N$)  It\^o process $(\xi_t)_{t\ge 0}$ with $\xi_0=0$ and
\begin{equation}\label{Ito_process}
\dd \xi_t=\beta_t\dd t+\dd W_t, \quad t\ge0,
\end{equation}
where $W$ is a Wiener process in $\R^d$, and $(\beta_t)_{t\ge 0}$ is  a {\blue progressively measurable  process.}
Denote by $\mu_\xi$ the law of the process $\xi$ in $\C([0,\infty), \R^d)$, and by  $\mu_W$ the law of $W$; the latter will be also called the \emph{Wiener measure} on $\C([0,\infty), \R^d)$.

 \begin{Theorem}\label{tKLGirsanov}
 $$
 D_{KL}(\mu_\xi\|\mu_W)\leq {\frac12}\E\int_0^\infty |\beta_t|^2\, \dd t.
 $$
 \end{Theorem}
 {\blue
 This result is not new,  it is essentially \cite[Theorem 8]{Us} up to a minor difference which will be discussed later in Remark \ref{rA30}. However, apparently it is not widely known, hence for the convenience of the reader we give its proof together with a short discussion. } Recall that the Girsanov theorem  (see, e.g., \cite[Theorem~6.2]{LS}) states that $\xi$ is a Wiener process w.r.t. the new probability measure
 $\dd \Q:=\mathcal{E}\dd \P$ with
 $$
 \mathcal{E}:=\exp\left(-\int_0^\infty \beta_t\, \dd W_t-{\frac12}\int_0^\infty|\beta_t|^2\, \dd t\right)
 $$
under the assumption that
\begin{equation}\label{Girsanov_one}
 \E  \mathcal{E}=1.
 \end{equation}{\blue
To verify this assumption, the \emph{Novikov condition} is most commonly used:}
\begin{equation}\label{beta_Novikov}
\E\exp\left({\frac12}\int_0^\infty|\beta_t|^2\, \dd t\right)<\infty.
\end{equation}
This condition is non-improvable in the sense that the similar condition
$$
\E\exp\left(\left({\frac12}-\eps\right)\int_0^\infty|\beta_t|^2\, \dd t\right)<\infty
$$
with any $\eps>0$ is not sufficient for \eqref{Girsanov_one}; see \cite[Example 6, Chapter 6.2]{LS}.  On the other hand,  under a much weaker condition
$$
\P\Bigl(\int_0^\infty|\beta_t|^2\, \dd t<\infty\Bigr)=1
$$
{\blue the absolute continuity  $\mu_\xi\ll \mu_W$ holds (see, e.g., \cite[Theorem~7.4]{LS}), though not much information about the Radon-Nikodym density $\dd \mu_\xi/\dd\mu_W$ is available. Theorem~\ref{tKLGirsanov} has the same virtue with the latter result: under a weak condition
\be\label{beta_square}
\E \int_0^\infty|\beta_t|^2\, \dd t<\infty
\ee
a bound for the KL-divergence of $\mu_\xi$ from $\mu_W$ is given without specifying $\dd \mu_\xi/\dd\mu_W$. The proof is  based on a similar localization argument.
}

 \begin{proof}[Proof of Theorem~\ref{tKLGirsanov}]
We begin by observing that if condition \eqref{beta_square} is not satisfied, then the statement of the theorem trivially holds. Thus from now on we can assume  \eqref{beta_square}.

 Assume first that additionally \eqref{beta_Novikov} holds true. Then \eqref{Girsanov_one} is satisfied, see, e.g., \cite[Theorem 6.1]{LS}. Hence  $\mu_\xi\sim\mu_W$, and the function
 $$
 p(x):=\frac{\dd \mu_W}{\dd \mu_\xi}(x), \quad x\in \C([0,\infty), \R^d)
 $$
 satisfies
\begin{equation*}
p(\xi)=\E[ \mathcal{E}| \sigma(\xi)];
\end{equation*}
see \cite[Theorem~7.3]{LS}. Therefore
\begin{align}\label{Jensen}
D_{KL}(\mu_\xi\|\mu_W)&=-\int_{\C([0,\infty))} \log\frac{\dd \mu_W}{\dd \mu_\xi}\, \dd \mu_\xi=-\int_{\C([0,\infty))} \log p \,\dd \mu_\xi\nn \\
&=-\E\log p(\xi) =-\E\log\bigl( \E[ \mathcal{E}| \sigma(\xi)] \bigr)\le -\E \log \mathcal{E},
\end{align}
where  in the last inequality we have used the conditional Jensen inequality. Assumption \eqref{beta_Novikov} yields \eqref{beta_square} and thus
$$
\E \int_0^\infty \beta_t\, \dd W_t=0.
$$
Therefore
$$
-\E \log \mathcal{E}= \frac12\E\int_0^\infty |\beta_t|^2\, \dd t,
$$
which gives the required bound.

Now we can treat the general case.  Define a sequence of stopping times
$$
\tau_n:=\inf\Bigl\{t: \int_0^t |\beta_s|^2\, \dd s\geq n\Bigr\}, \quad n\in\Z_+.
$$
Let $\beta^n_t:=\beta_t\I_{t\leq \tau_n}$. Consider the corresponding It\^o process
$$
\dd \xi_t^n=\beta_t^n\dd t+\dd W_t, \quad \xi_0^n=0
$$
and denote its law by $\mu_{\xi^n}$. By definition, the process $\beta^n$ satisfies the Novikov condition \eqref{beta_Novikov}. Thus, by the first part of the proof
\begin{equation}\label{KLapprox}
D_{KL}(\mu_{\xi^n}\|\mu_W)\le \frac12 \E\int_0^{\tau_n} |\beta_t|^2\, \dd t\le\frac12\E\int_0^\infty |\beta_t|^2\, \dd t.
\end{equation}

Note that the processes $\xi^n$ and $\xi$ coincide on the set $\{\tau_n=\infty\}$. On the other hand, condition \eqref{beta_square} implies
$$
\P(\tau_n<\infty)\leq \P\left(\int_0^\infty|\beta_t|^2\, \dd t\geq n\right)\to 0 \quad \text{as $n\to \infty$}.
$$
Hence
\begin{equation}\label{TVconvergence}
d_{TV}(\mu_{\xi^n},\mu_\xi)\to0\quad\text{as $n\to \infty$}.
\end{equation}
Denote  $f_n:=\frac{\dd \mu_{\xi^n}}{\dd \mu_W}$, $f:=\frac{\dd \mu_{\xi}}{\dd \mu_W}$. It follows from \eqref{TVconvergence} that $f_n\to f$ in $L_1(\mu_W)$ as $n\to\infty$. Therefore $f_n\log f_n$ converges in measure $\mu_W$ to $f\log f$ as $n\to\infty$. Since the sequence $(f_n\log f_n)_{n\in\Z_+}$ is uniformly bounded from below by a constant, an application of Fatou's lemma and \eqref{KLapprox} give
$$
D_{KL}(\mu_\xi\|\mu_W)=\int_{\C([0,\infty))} f\log f \dd\mu_W\le \liminf_{n\to \infty} \int_{\C([0,\infty))} f_n\log f_n \dd \mu_W\le \frac12\E\int_0^\infty |\beta_t|^2\, \dd t.
$$
This completes the proof of the theorem.
 \end{proof}
 {\blue
\begin{Remark}\label{rA30} A minor difference between Theorem~\ref{tKLGirsanov} and \cite[Theorem 8]{Us} is that the latter one, prior to the localization argument, uses invertibility of Wiener maps, and thus requires $\beta$ to be progressively measurable w.r.t. the natural filtration of $W$. In the above proof this step is made using the Jensen inequality \eqref{Jensen}, hence this hidden limitation is not involved.
\end{Remark}
}

 \begin{Remark}\label{rA3} The proof of Theorem~\ref{tKLGirsanov}  without any substantial changes can be transferred to
 the infinite-dimensional setting. Namely, using \cite[Theorem 10.14]{DaPratoZab} instead of the classical Girsanov theorem, one gets essentially the same bound for ${\blue D_{KL}(\mu_\xi\|\mu_W)}$ in the setting where $\xi, W$ in \eqref{Ito_process} are cylindrical processes in a Hilbert space $H$. In this case, $\mu_\xi$, $\mu_W$ are \emph{cylindrical} measures; this requires  some additional technical steps, e.g., the specification of the Radon-Nikodym derivative for cylindrical measures. In order to keep the exposition relatively simple and clear, we do not develop this possibility in further details here. {\blue Note also that simply replacing $\beta_t$ by $\beta_t1_{t\leq T}$ one can get the same statement for a finite time interval $[0,T]$.}
 \end{Remark}

As mentioned before, Theorem~\ref{tKLGirsanov} provides a non-trivial bound only if condition \eqref{beta_square} holds. Now let us further relax this condition and assume that only for some $\delta\in(0,1)$
\begin{equation}\label{beta_delta}
M_{\delta}:=\E\Bigl(\int_0^\infty |\beta_t|^2\, \dd t\Bigr)^\delta<\infty.
\end{equation}
In this case it is also possible to measure the difference between $\mu_\xi$ and $\mu_W$; the price to pay is that this difference will be expressed in terms of total variation distance rather than KL--divergence.

\begin{Theorem}\label{L:e2e4} For any $\delta\in(0,1)$ we have
\begin{align}
\label{TV_bound_one_delta}
&d_{TV}(\mu_\xi, \mu_W)\leq 2^{(1-\delta)/(1+\delta)}(M_{\delta})^{1/(1+\delta)};\\
&\label{TV_bound_two_delta}
d_{TV}(\mu_\xi, \mu_W)\leq 1-\frac16\min\bigl(\frac18, e^{- (2^{2-\delta}M_{\delta})^{1/\delta}}\bigr).
\end{align}
Further, for any measurable set $A\subset \C([0, \infty), \R^d)$ and any $N>1$
\begin{equation}\label{lower_bound_Ito_delta}
\mu_W(A)\ge \frac1N\bigl(\mu_\xi(A)-\frac{2^{1-\delta}M_{\delta}}{(\log N)^\delta}-\frac{\log 2}{\log N}\bigr).
\end{equation}
\end{Theorem}

\begin{proof}
Fix  $\delta\in(0,1)$. As before we  can assume that condition \eqref{beta_delta} is satisfied; otherwise all the bounds of the theorem hold trivially.

We begin with the proof of \eqref{TV_bound_one_delta}.  Let $\tau_n$, $\xi^n$ and $\mu_{\xi^n}$, $n\in\Z_+$ be the same as in the proof of Theorem \ref{tKLGirsanov}. By Theorem~\ref{tKLGirsanov} we have for $n\in\Z_+$
\begin{equation}\label{boundtvstep0}
D_{KL}(\mu_{\xi^n}\|\mu_W)\le \frac12\E\int_0^{\tau_n} |\beta_t|^2\, \dd t=\frac12\E\min \Bigl(\int_0^{\infty} |\beta_t|^2\, \dd t, n\Bigr)\le \frac12n^{1-\delta}M_{\delta}.
\end{equation}
In addition,
\begin{equation}\label{boundtvstep0.5}
d_{TV}(\mu_\xi, \mu_{\xi^n})\leq \P\left(\tau_n<\infty\right)=\P\Bigl(\int_0^{\infty} |\beta_t|^2\, \dd t\geq n\Bigr)\leq n^{-\delta}M_{ \delta}.
\end{equation}
Then it follows by \eqref{otzenka2} that
$$
d_{TV}(\mu_\xi, \mu_W)\le d_{TV}(\mu_\xi, \mu_{\xi^n})+d_{TV}(\mu_{\xi^n}, \mu_W) \le
n^{-\delta}M_{ \delta}+\frac12n^{1/2-\delta/2}\sqrt{M_\delta}.
$$
Taking $n=(4M_{ \delta})^{1/(1+\delta)}$, we obtain \eqref{TV_bound_one_delta}.

To establish \eqref{lower_bound_Ito_delta} we apply inequality \eqref{diffbound} to the measures $\mu_W$ and $\mu_{\xi^n}$, $n>1$. For any $N>1$ and any measurable set $A$ we get
\begin{equation}\label{boundtvstep1}
\mu_W(A)\ge\frac1N \mu_{\xi^n}(A)-\frac{ D_{KL}(\mu_{\xi^n}\|\mu_W)+\log2}{N\log N}\ge
\frac1N \mu_{\xi^n}(A)-\frac{ n^{1-\delta}M_{\delta}+2\log2}{2N\log N},
\end{equation}
where we have also used \eqref{boundtvstep0}. Further, inequality \eqref{boundtvstep0.5} implies
\begin{equation*}
\mu_{\xi^n}(A)\ge \mu_{\xi}(A)-d_{TV}(\mu_\xi, \mu_{\xi^n})\ge \mu_{\xi}(A)-n^{-\delta}M_{ \delta}.
\end{equation*}
Combining this with \eqref{boundtvstep1} we obtain
$$
\mu_W(A)\geq\frac1N\bigl(\mu_\xi(A)-\frac{n^{1-\delta}M_{\delta}}{2\log N}-n^{-\delta}M_{\delta}-\frac{\log 2}{\log N}\bigr).
$$
By choosing $n:=2 \log N$, we obtain \eqref{lower_bound_Ito_delta}.

Finally, to derive \eqref{TV_bound_two_delta} we apply  \eqref{lower_bound_Ito_delta} with $\log N:=\max(3\log 2, (2^{2-\delta}M_{\delta})^{1/\delta})$. We get
$$
\mu_W(A)\geq\frac1N\bigl(\mu_\xi(A)-\frac56\bigr).
$$
This yields
\begin{align*}
 d_{TV}(\mu_\xi, \mu_W)&=\sup_A \bigl(\mu_\xi(A)-\mu_W(A)\bigr)\leq
\frac5{6N}+\sup_A (1-\frac1N)\mu_\xi(A)\\
&\le1-\frac1{6N}\le 1-\frac16\min\bigl(\frac18, e^{- (2^{2-\delta}M_{\delta})^{1/\delta}}\bigr).
\end{align*}
This completes the proof of \eqref{TV_bound_two_delta}.
\end{proof}

\end{document}